\documentclass[12pt]{amsart}
\usepackage{amsmath,amssymb,amsthm,amsfonts,mathrsfs,stmaryrd}
\usepackage{amscd}
\usepackage[usenames,dvipsnames]{color}
\usepackage{hyperref}
\usepackage{graphicx,subfigure}
\definecolor{nicegreen}{RGB}{0,180,0}
\hypersetup{colorlinks=true,citecolor=nicegreen,linkcolor=Red,urlcolor=blue,pdfstartview=FitH}

\usepackage[divide={2.45cm,*,2.45cm}]{geometry}
\usepackage{enumerate}
\usepackage{color}
\usepackage[all]{xy}

\newtheorem{theo}{Theorem}[section]
\newtheorem{coro}[theo]{Corollary}
\newtheorem{lemm}[theo]{Lemma}
\newtheorem{prop}[theo]{Proposition}

\theoremstyle{remark}

\theoremstyle{plain}
\newtheorem*{thmI}{Theorem (Corollary \ref{scbij})}
\newtheorem*{defnI}{Definition \ref{corr}, Modified}

\theoremstyle{definition}
\newtheorem{defi}[theo]{Definition}

\newcommand{\qp}{\mathbb{Q}_p}
\newcommand{\qpp}{\mathbb{Q}_{p^2}}

\newcommand{\zp}{\mathbb{Z}_p}
\newcommand{\zpp}{\mathbb{Z}_{p^2}}
\newcommand{\fpb}{\ol{\mathbb{F}}_p}
\newcommand{\pp}{\mathfrak{p}}
\newcommand{\oo}{\mathfrak{o}}
\newcommand{\hh}{\mathcal{H}}
\newcommand{\T}{\textnormal{T}}

\newcommand{\s}{\textnormal{S}}
\newcommand{\bpi}{\boldsymbol{\pi}}
\newcommand{\gal}{\mathcal{G}}
\newcommand{\ii}{\mathcal{I}}
\newcommand{\cusp}[2]{(\omega^{#1}\circ\det)\otimes\bpi_{#2}}
\newcommand*{\longhookrightarrow}{\ensuremath{\lhook\joinrel\relbar\joinrel\rightarrow}}
\newcommand*{\longtwoheadrightarrow}{\ensuremath{\relbar\joinrel\twoheadrightarrow}}

\newcommand{\ol}[1]{\overline{#1}}
\newcommand{\gm}{\mathbf{G}_m}
\newcommand{\GG}{\mathfrak{G}}
\newcommand{\HH}{\mathfrak{H}}

\begin{document}
\nocite{*}

\title{A Classification of the Irreducible mod-$p$ Representations of $\textnormal{U}(1,1)(\mathbb{Q}_{p^2}/\mathbb{Q}_p)$}
\date{\today}
\author{Karol Kozio\l}
\address{Department of Mathematics, University of Toronto, Toronto, ON M5S 2E4  CANADA} \email{karol@math.toronto.edu}

\begin{abstract}
 Let $p$ be a prime number.  We classify all smooth irreducible mod-$p$ representations of the unramified unitary group $\textnormal{U}(1,1)(\mathbb{Q}_{p^2}/\mathbb{Q}_p)$ in two variables.  We then investigate Langlands parameters in characteristic $p$ associated to $\textnormal{U}(1,1)(\mathbb{Q}_{p^2}/\mathbb{Q}_p)$, and propose a correspondence between certain equivalence classes of Langlands parameters and certain isomorphism classes of semisimple $L$-packets on $\textnormal{U}(1,1)(\mathbb{Q}_{p^2}/\mathbb{Q}_p)$. 
 \end{abstract}

\maketitle
\tableofcontents

\section{Introduction}

Recently, the mod-$p$ representation theory of $p$-adic reductive groups has garnered a great deal of attention as a result of its roles in the mod-$p$ and $p$-adic Local Langlands Programs.  The expectation is that there exists a matching between (packets of) smooth mod-$p$ representations of a $p$-adic reductive group and certain Galois representations.  Representations of the group $\textnormal{GL}_2(\qp)$ have been widely studied and analyzed, and a (semisimple) mod-$p$ Local Langlands Correspondence has been established by Breuil (\cite{Br03}) based on the explicit determination of the irreducible mod-$p$ representations of $\textnormal{GL}_2(\qp)$.  Moreover, this correspondence is compatible with the $p$-adic Local Langlands Correspondence (cf. \cite{Br03b}; see also \cite{Br04}, \cite{Co10}, \cite{Em10}, \cite{Ki09}, \cite{Ki10}, \cite{Pas10}).

The smooth irreducible mod-$p$ representations of $\textnormal{SL}_2(\qp)$ have recently been classified by Abdellatif in \cite{Ab12}, by examining restrictions of the irreducible representations of $\textnormal{GL}_2(\qp)$.  This allowed her to take the first steps towards a mod-$p$ Local Langlands Correspondence for $\textnormal{SL}_2(\qp)$.  In addition, her results are the first to consider a mod-$p$ Local Langlands Correspondence with $L$-packets.

Several aspects of the mod-$p$ representation theory of unitary groups have already been considered by Abdellatif in \cite{Ab11}, and by the author and Xu in \cite{KX12}.  In the present article, we utilize the work of Breuil and Abdellatif to investigate the smooth irreducible mod-$p$ representations of the unitary group $\textnormal{U}(1,1)(\qpp/\qp)$, where $\qpp$ denotes the unramified quadratic extension of the field of $p$-adic numbers $\qp$.  The irreducible subquotients of representations of $\textnormal{U}(1,1)(\qpp/\qp)$ parabolically induced from characters have been classified by Abdellatif in \cite{Ab13}.  We shall be interested in representations which do not arise in this way, which we will refer to as \emph{supercuspidal} representations (we will comment on terminology at the end of this introduction).  These representations are the ones which are expected to play a central role in a potential Local Langlands Correspondence.

We begin our investigation in a more general context.  Denote by $F$ a nonarchimedean local field of residual characteristic $p$, and $E$ an unramified quadratic extension.  Let $G$ be the group $\textnormal{U}(1,1)(E/F)$, $G_\s = \textnormal{SU}(1,1)(E/F)$ its derived subgroup, and $I_\s(1)$ the unique pro-$p$ Sylow subgroup of the standard Iwahori subgroup of $G_\s$.  The \emph{pro-$p$-Iwahori-Hecke algebra} $\hh_{\fpb}(G_\s,I_\s(1))$ is the convolution algebra of compactly supported, $\fpb$-valued functions on the double coset space $I_\s(1)\backslash G_\s/I_\s(1)$.  As $G_\s\cong \textnormal{SL}_2(F)$, the structure and properties of $\hh_{\fpb}(G_\s,I_\s(1))$ are well-understood (cf. \cite{Ab11}, \cite{Vig05}).  In particular, a classification of finite-dimensional simple right modules for $\hh_{\fpb}(G_\s,I_\s(1))$ is known (see \cite{Ab11}).  We review the necessary results in Chapter 3.


The results of \cite{Ab13} provide a classification of smooth irreducible nonsupercuspidal representations of any connected quasisplit reductive group of relative rank 1.  We make the computations explicit in Section 4, and obtain an explicit description of all irreducible nonsupercuspidal representations of $G$ (Theorem \ref{nonscclass}).  We then investigate the behavior of irreducible representations upon restriction to the derived subgroup $G_\s$.

Next, we specialize to the case where $F = \qp$ and $E = \qpp$.  Under these assumptions, the smooth irreducible supercuspidal representations of $\textnormal{GL}_2(\qp)$ and $\textnormal{SL}_2(\qp)$ have been classified by Breuil and Abdellatif, respectively (cf. \cite{Br03}, \cite{Ab12}).  Using the algebra $\hh_{\fpb}(G_\s,I_\s(1))$ and a cohomological argument, we show in Section 5 that the supercuspidal representations of $\textnormal{SL}_2(\qp)\cong G_\s$ lift to smooth irreducible representations of $G$, which we denote 
$$\cusp{k}{r}$$
for $0\leq r \leq p - 1$ and $0\leq k < p + 1$ (see Definition \ref{defofsc}).  Moreover, we show that every smooth irreducible supercuspidal representation of $\textnormal{U}(1,1)(\qpp/\qp)$ is of this form (Theorem \ref{scclass}), and thereby obtain a classification of all smooth irreducible representations (Corollary \ref{u11class}).  To conclude, we arrange the irreducible representations into sets called \emph{$L$-packets}, and determine the $L$-packets on $\textnormal{U}(1,1)(\qpp/\qp)$ explicitly.

In the final section, we define the relevant Galois groups and $L$-groups attached to $G$.  Our definitions are adapted from the complex setting (see \cite{Ro90} for the classical definitions).  Thus, we are led to investigate Langlands parameters associated to the group $G$, that is, certain homomorphisms
$$\varphi:\textnormal{Gal}(\ol{\mathbb{Q}}_p/\qp)\longrightarrow {}^LG = \textnormal{GL}_2(\fpb)\rtimes\textnormal{Gal}(\ol{\mathbb{Q}}_p/\qp).$$
Our first (somewhat surprising) result in this direction is Proposition \ref{nostables}, which asserts that there do not exist any parameters $\varphi$ such that the Galois representation associated to $\varphi|_{\textnormal{Gal}(\ol{\mathbb{Q}}_p/\qpp)}$ is irreducible.

The aforementioned result suggests that the Langlands parameters associated to $G$ which are of interest are ``reducible'' in some sense.  We therefore consider homomorphisms whose image lies in the $L$-groups ${}^LJ$ and ${}^LT$, where $J$ is the unique \emph{elliptic endoscopic group} associated to $G$, and $T$ is the maximal torus of $G$.  We classify all such parameters in Propositions \ref{u1params} and \ref{paramsfromT}, and determine the possible equivalences among the parameters (Lemmas \ref{paramequiv}, \ref{paramequiv2}, and \ref{endvsnonend}).

In the complex setting, the parameters coming from ${}^LJ$ play a pivotal role in the representation theory of $G$ (cf. \cite{Ro90}).  In the characteristic $p$ setting, they remain of particular importance:
\begin{thmI}
Suppose $0\leq k,\ell < p + 1$ and $k\neq \ell$.  There exists a bijection between equivalence classes of Langlands parameters factoring through the group ${}^LJ$ and $L$-packets of irreducible supercuspidal representations of the group $G = \textnormal{U}(1,1)(\qpp/\qp)$, given by
$$\varphi_{k,\ell}\longleftrightarrow \{\cusp{\ell}{[k - \ell - 1]},~\cusp{k}{[\ell - k - 1]}\},$$
where $[k - \ell - 1]$ (resp. $[\ell - k - 1]$) denotes the unique integer between $0$ and $p - 1$ equivalent to $k - \ell - 1$ (resp. $\ell - k - 1$) modulo $p + 1$.  Moreover, this bijection is compatible with twisting by characters on both sides.  
\end{thmI}
We refer to Definition \ref{endparams} for the precise description of the parameters $\varphi_{k,\ell}$.

Finally, we extend this bijection to include semisimple nonsupercuspidal representations (Definition \ref{corr}), using the parameters coming from ${}^LT$.  This allows us to make explicit a case of \emph{endoscopic transfer} from irreducible representations of $J$ to semisimple $L$-packets on $G$ (see the remarks following Definition \ref{corr}).  

\vspace{\baselineskip}

\noindent\textbf{Remark on Terminology.}  We briefly address our choice of nomenclature.  The notion of supersingularity was introduced by Barthel and Livn\'{e} (\cite{BL94} and \cite{BL95}) in their classification of smooth irreducible mod-$p$ representations of $\textrm{GL}_2(F)$.  For a general connected reductive group, a smooth irreducible admissible representation $\pi$ is called \emph{supersingular} if the Hecke eigenvalues of $\pi$ are ``as null as possible,'' while $\pi$ is called \emph{supercuspidal} if it is not a subquotient of a representation parabolically induced from a smooth irreducible admissible representation of a proper Levi subgroup.  Thanks to recent work of Abe--Henniart--Herzig--Vign\'{e}ras (\cite{AHHV}), we now know that these notions are equivalent; we will use them interchangeably.

\vspace{\baselineskip}

\noindent\textbf{Acknowledgements.}   I would like to thank my advisor Rachel Ollivier, for many enlightening discussions throughout the course of working on this article, as well as her encouragement and advice.  I would also like to thank Ramla Abdellatif, Florian Herzig, and Shaun Stevens for many helpful comments and discussions.  During the preparation of this article, support was provided by NSF Grant DMS-0739400.

\vspace{\baselineskip}

\section{Notation}\label{notation}
Fix a prime number $p$, and let $F$ be a nonarchimedean local field of residual characteristic $p$.  We assume throughout that the characteristic of $F$ is not equal to 2 (so that if $p = 2$, $F$ is a finite extension of $\mathbb{Q}_2$).  Denote by $\oo_F$ its ring of integers, and by $\pp_F$ the unique maximal ideal of $\oo_F$.  Fix a uniformizer $\varpi_F$ and let $k_F = \oo_F/\pp_F$ denote the residue field of size $q$, a power of $p$.  We fix also a separable closure $\overline{F}$ of $F$, and let $k_{\ol{F}}$ denote its residue field.  

Let $E$ denote the unique unramified extension of degree 2 in $\ol{F}$.  We denote by $\oo_E, \pp_E$, etc., the analogous objects for $E$.  Since $E$ is unramified, we may and do take $\varpi_E = \varpi_F =: \varpi$ as our uniformizer.  Let $\iota:k_{\ol{F}}\stackrel{\sim}{\longrightarrow} \fpb$ denote a fixed isomorphism, and assume that every $\fpb^\times$-valued character factors through $\iota$.  We identify $k_F$ and $k_E$ with $\mathbb{F}_q$ and $\mathbb{F}_{q^2}$, respectively, using the isomorphism $\iota$.  We will also identify $\mathbb{F}_{q^2}^\times$ with the image of the Teichm\"{u}ller lifting map $[\ \cdot\ ]:\mathbb{F}_{q^2}^\times\longrightarrow \oo_E^\times$ when convenient.  

We let $x\longmapsto\ol{x}$ denote the nontrivial Galois automorphism of $E$ fixing $F$.  This automorphism preserves $\oo_E$ and $\pp_E$, and induces the automorphism $x\longmapsto x^q$ on $\mathbb{F}_{q^2}$.  We take $\epsilon\in \oo_F^\times$ to be a fixed element for which $E = F(\sqrt{\epsilon})$, so that $\ol{\sqrt{\epsilon}} = -\sqrt{\epsilon}$.  We define $\textnormal{U}(1)(E/F)$ to be the kernel of the norm map
\begin{center}
 \begin{tabular}{rccl}
  $\textnormal{N}_{E/F}:$ & $E^\times$ & $\longrightarrow$ & $F^\times$\\
 & $x$ & $\longmapsto$ &  $x\ol{x}$.
 \end{tabular}
\end{center}
The norm map induces an isomorphism
$$\oo_E^\times/\oo_F^\times \textnormal{U}(1)(E/F)\stackrel{\sim}{\longrightarrow} \oo_F^\times/(\oo_F^\times)^2 \cong \begin{cases} \mathbb{Z}/2\mathbb{Z} & p\neq 2,\\ (\mathbb{Z}/2\mathbb{Z})^{\oplus [F:\mathbb{Q}_2] + 1} & p = 2.\end{cases}$$
This follows from, for example, Proposition II.5.7 of \cite{Neu92}.  When $p \neq 2$ (resp. $p = 2$), we denote by $\vartheta$ (resp. $\vartheta_1,\ldots,\vartheta_{[F:\mathbb{Q}_2] + 1}$) a fixed element (resp. fixed elements) in $\oo_E^\times$ whose image (resp. images) in $\oo_E^\times/\oo_F^\times\textnormal{U}(1)(E/F)$ is a generator (resp. gives a set of generators).  

Denote by $G$ the $F$-rational points of the algebraic group $\mathbf{U}(1,1)$, defined and quasisplit over $F$.  Explicitly, we take $G$ to have the form
$$G = \left\{g\in \textrm{GL}_2(E): g^*\begin{pmatrix} 0 & 1\\ 1 & 0\end{pmatrix}g = \begin{pmatrix} 0 & 1\\ 1 & 0\end{pmatrix}\right\},$$
where $g^* = \ol{g}^\top$ denotes the conjugate transpose of a matrix $g$ with coefficients in $E$.

Let $K$ denote the maximal compact subgroup of $G$ given by 
$$K := \textrm{GL}_2(\oo_E)\cap G,$$
(cf. \cite{Ro90}, Section 1.10) and let 
$$K_1 :=\begin{pmatrix}1 + \pp_E & \pp_E \\ \pp_E & 1 + \pp_E\end{pmatrix}\cap G$$
denote its pro-$p$ radical.  We define 
$$\Gamma := K/K_1 \cong \textnormal{U}(1,1)(\mathbb{F}_{q^2}/\mathbb{F}_{q}).$$
The Iwahori subgroup $I$ is defined as the preimage under the quotient map $K\longtwoheadrightarrow \Gamma$ of the Borel subgroup of upper triangular matrices in $\Gamma$.  We denote by $I(1)$ the pro-$p$ radical of $I$, which is the preimage of the upper triangular unipotent elements of $\Gamma$.  Explicitly, we have
$$I := \begin{pmatrix}\oo_E^\times & \oo_E \\ \pp_E & \oo_E^\times\end{pmatrix}\cap G,\qquad I(1):=\begin{pmatrix}1 + \pp_E & \oo_E \\ \pp_E & 1 + \pp_E \end{pmatrix}\cap G.$$

Let $B$ denote the Borel subgroup of upper triangular elements of $G$, $U$ its unipotent radical, and $U^-$ the opposite unipotent.  The subgroups $U$ and $U^-$ are both isomorphic to the additive group $F$.

We let $G_\textrm{S}$ denote the derived subgroup of $G$.  For any subgroup $J$ of $G$, we let $J_\s$ denote its intersection with $G_\s$.  We have $G_\textrm{S} = \textnormal{SU}(1,1)(E/F)\cong \textnormal{SL}_2(F)$, the latter isomorphism given by conjugation by the element $$\begin{pmatrix}\sqrt{\epsilon} & 0 \\ 0 & 1 \end{pmatrix}\in \textnormal{GL}_2(E).$$  We shall exploit this isomorphism to give a classification of the smooth irreducible representations of $\textnormal{U}(1,1)(\mathbb{Q}_{p^2}/\mathbb{Q}_p)$.

The maximal torus $T$ of $G$ consists of all elements of the form 
$$\begin{pmatrix}a & 0 \\ 0 & \ol{a}^{-1}\end{pmatrix},$$ 
with $a\in E^\times$.  Note that $T$ (or more precisely, the algebraic group defining $T$) is not split over $F$.  The maximal torus of $G_\s$ is $T_\s = T\cap G_\s$; it consists of all elements of the form 
$$\begin{pmatrix}a & 0 \\ 0 & a^{-1}\end{pmatrix},$$ 
with $a\in F^\times$, and is split over $F$.  The center $Z$ of $G$ is given by the subgroup of elements 
$$\begin{pmatrix}a & 0 \\ 0 & a\end{pmatrix},$$
with $a\in \textnormal{U}(1)(E/F)$.  We let 
$$T_0 := T\cap K\qquad\textnormal{and}\qquad T_1:= T\cap K_1.$$

Finally, we define the following distinguished elements of $\textnormal{GL}_2(E)$: 
$$\displaystyle{s:=\begin{pmatrix}0 &  1\\  1 & 0 \end{pmatrix}}, \qquad \displaystyle{\beta := \begin{pmatrix}0 & 1 \\ \varpi & 0 \end{pmatrix}},$$

$$\displaystyle{n_s:=\begin{pmatrix}0 & -\sqrt{\epsilon}^{-1}\\ \sqrt{\epsilon}& 0\end{pmatrix}},\qquad\displaystyle{n_{s'}:=\begin{pmatrix}0 & -\varpi^{-1}\sqrt{\epsilon}^{-1}\\\varpi\sqrt{\epsilon} & 0\end{pmatrix}},$$

$$\displaystyle{\theta := \begin{pmatrix}\vartheta & 0 \\ 0 & \ol{\vartheta}^{-1}\end{pmatrix}}~\textnormal{if}~p\neq 2,\qquad \displaystyle{\theta_i := \begin{pmatrix}\vartheta_i & 0 \\ 0 & \ol{\vartheta_i}^{-1}\end{pmatrix}}~\textnormal{if}~p = 2.$$

\section{Hecke Algebras}

In the course of determining the smooth irreducible representations of $\textnormal{U}(1,1)(\qpp/\qp)$, we shall make essential use of the pro-$p$-Iwahori-Hecke algebra of the derived subgroup $\textnormal{SU}(1,1)(\qpp/\qp)$.  We collect the relevant results here.  

\vspace{\baselineskip}
\subsection{Preliminaries}  We will be interested in the category $\mathfrak{Rep}_{\fpb}(G)$ of smooth representations of $G$ over $\fpb$.  We briefly recall some preliminary terminology.  Let $J$ be a closed subgroup of $G$, and let $(\sigma,V_\sigma)$ be a smooth $\fpb$-representation of $J$ (meaning that stabilizers are open).  We denote by $\textnormal{ind}_J^{G}(\sigma)$ the space of functions $f:G\longrightarrow V_\sigma$ such that $f(jg) = \sigma(j)f(g)$ for $j\in J, g\in G$, and such that the action of $G$ given by right translation is smooth (meaning that there exists some open subgroup $J'$, depending on $f$, such that $f(gj') = f(g)$ for every $j'\in J', g\in G$).  We let $\textnormal{c-ind}_J^{G}(\sigma)$ denote the subspace of $\textnormal{ind}_J^{G}(\sigma)$ spanned by functions whose support in $J\backslash G$ is compact.  These functors are called \emph{induction} and \emph{compact induction}, respectively.  We will mostly be concerned with the cases when $J$ is a compact open subgroup, or 
when $J$ is the group of $F$-rational points of a parabolic subgroup of $\mathbf{U}(1,1)$.  

\vspace{\baselineskip}
\subsection{Pro-$p$-Iwahori-Hecke Algebra}  Let $\pi$ be a smooth $\fpb$-representation of the group $G_\s = \textnormal{SU}(1,1)(E/F)$.  Frobenius Reciprocity for compact induction gives $$\pi^{I_\s(1)}\cong \textrm{Hom}_{I_\s(1)}(1,\pi|_{I_\s(1)})\cong \textrm{Hom}_{G_\s}(\textrm{c-ind}_{I_\s(1)}^{G_\s}(1),\pi),$$ where $1$ denotes the trivial character of $I_\s(1)$.  The \emph{pro-$p$-Iwahori-Hecke algebra} $$\hh_{\fpb}(G_\s,I_\s(1)) := \textnormal{End}_{G_\s}(\textnormal{c-ind}_{I_\s(1)}^{G_\s}(1))$$ is the algebra of ${G_\s}$-equivariant endomorphisms of the universal module $\textnormal{c-ind}_{I_\s(1)}^{G_\s}(1)$.  This algebra naturally acts on $\textrm{Hom}_{G_\s}(\textrm{c-ind}_{I_\s(1)}^{G_\s}(1),\pi)$ by pre-composition, which induces a right action on $\pi^{I_\s(1)}$.  In this way, we obtain the functor of $I_\s(1)$-invariants, $\pi\longmapsto\pi^{I_\s(1)}$, from the category of smooth $\fpb$-representations of ${G_\s}$ to the category of right $\hh_{\fpb}({G_\s},I_\s(1))$-modules.  

By adjunction, we have a natural identification 
$$\hh_{\fpb}({G_\s},I_\s(1)) \cong \textnormal{c-ind}_{I_\s(1)}^{G_\s}(1)^{I_\s(1)},$$ 
so we may view endomorphisms of $\textnormal{c-ind}_{I_\s(1)}^{G_\s}(1)$ as compactly supported functions on ${G_\s}$ which are $I_\s(1)$-biinvariant.  This leads to the following definition.

\begin{defi}
 Let $g\in {G_\s}$.  We let $\T_g\in \hh_{\fpb}({G_\s},I_\s(1))$ denote the endomorphism of $\textnormal{c-ind}_{I_\s(1)}^{G_\s}(1)$ corresponding by adjunction to the characteristic function of ${I_\s(1)gI_\s(1)}$; in particular, $\T_g$ maps the characteristic function of $I_\s(1)$ to the characteristic function of $I_\s(1)gI_\s(1)$.  
\end{defi}

Using the isomorphisms above, we see that if $\pi$ is a smooth $\fpb$- representation of ${G_\s}$, $v\in \pi^{I_\s(1)}$, and $g\in G_\s$, then 
\begin{equation}\label{act}
v\cdot \T_g = \sum_{u \in I_\s(1)\backslash I_\s(1)gI_\s(1)}u^{-1}.v = \sum_{u\in I_\s(1)/I_\s(1)\cap g^{-1}I_\s(1)g} ug^{-1}.v. 
\end{equation}
For more details, see \cite{Vig05}.  

\vspace{\baselineskip}
\subsection{Supersingular Modules}  Let $H_\s := T_{0,\s}/T_{1,\s} \cong \mathbb{F}_q^\times$.  For $h\in H_\s$, we let $\T_h$ denote the operator $\T_{t_0}$ for any preimage $t_0$ of $h$ in $T_{0,\s}$.  Since the group $G_\s$ is \emph{split}, we may apply results of \cite{Vig05} to give the structure of $\hh_{\fpb}(G_\s,I_\s(1))$ (see \cite{Vig14} for the case of a general reductive group).  

\begin{theo}\label{gens}
The operators $\T_{n_s}$, $\T_{n_{s'}}$ and $\T_h$, for $h\in H_\s$, generate $\hh_{\fpb}(G_\s, I_\s(1))$ as an algebra.  
\end{theo}

\begin{proof}
 The claim follows from (the remark following) Theorem 1 of \cite{Vig05}.  
\end{proof}

In the study of Hecke modules over fields of characteristic $p$, the notion of supersingularity plays a prominent role.  We recall the definition here.  

\begin{defi}[\cite{Vig05}, Definition 3]\label{defofssing}
 Let $M$ be a nonzero simple right $\hh_{\fpb}(G_\s,I_\s(1))$-module which admits a central character.  We say $M$ is \emph{supersingular} if every element of the center of $\hh_{\fpb}(G_\s,I_\s(1))$ which is of ``positive length'' acts by 0.  For the precise notion of ``positive length,'' see the discussion preceding Definition 2 (\emph{loc. cit.}).  
\end{defi}

The finite-dimensional simple right $\hh_{\fpb}(G_\s, I_\s(1))$-modules have been classified in Chapitre 6 of \cite{Ab11}.  The supersingular modules take on a particularly simple form:

\begin{prop}\label{ssingmods}
 The supersingular $\hh_{\fpb}(G_\s,I_\s(1))$-modules are \linebreak one-dimensional.  They are given by:
\begin{center}
  \begin{tabular}{ccrccccclcccl}
   $M_{0}$ & : & $\T_h$ & $\longmapsto$ & $1$, & & $\T_{n_s}$ & $\longmapsto$ & $\phantom{-}0$, & & $\T_{n_{s'}}$ & $\longmapsto$ & $-1$;\\
   $M_{q-1}$ & : & $\T_h$ & $\longmapsto$ & $1$, & & $\T_{n_s}$ & $\longmapsto$ & $-1$, & & $\T_{n_{s'}}$ & $\longmapsto$ & $\phantom{-}0$;\\
   $M_{r}$ & : & $\T_h$ & $\longmapsto$ & $a^{-r}$, & & $\T_{n_s}$ & $\longmapsto$ & $\phantom{-}0$, & & $\T_{n_{s'}}$ & $\longmapsto$ & $\phantom{-}0$,
  \end{tabular}
 \end{center}
where $0 < r < q-1,$ $a\in \mathbb{F}_q^\times$, and $h = \left(\begin{smallmatrix}a & 0 \\ 0 & a^{-1}\end{smallmatrix}\right)\in H_\s$.  
\end{prop}

In order to make use of the machinery of Hecke modules, we will need a precise relationship between smooth representations of $\textnormal{SU}(1,1)(E/F)$ and $\hh_{\fpb}(G_\s,I_\s(1))$-modules.  The following theorems provide us with the necessary link.

\begin{theo}[\cite{Ab11}, Corollaire 6.1.10 (i)]\label{bij1}
 The functor of $I_\s(1)$- invariants $\pi\longmapsto \pi^{I_\s(1)}$ induces a bijection between isomorphism classes of smooth, irreducible, nonsupercuspidal representations of $\textnormal{SU}(1,1)(E/F)$ and isomorphism classes of simple, finite-dimensional, nonsupersingular right $\hh_{\fpb}(G_\s,I_\s(1))$-modules.  
\end{theo}

\begin{theo}[\cite{Ab11}, Corollaire 6.1.10 (ii)]\label{bij2}
 The functor of $I_\s(1)$- invariants $\pi\longmapsto \pi^{I_\s(1)}$ induces a bijection between isomorphism classes of smooth, irreducible representations of $\textnormal{SU}(1,1)(\qpp/\qp)$ and isomorphism classes of simple, finite-dimensional right $\hh_{\fpb}(G_\s,I_\s(1))$-modules.  Under this bijection, the supercuspidal representation $\pi_r$ (of Definition \ref{defofsl2ssing}) corresponds to the supersingular module $M_r$ (of Proposition \ref{ssingmods}).  
\end{theo}

\begin{coro}\label{twistisom}
 Let $\pi$ be a smooth, irreducible supercuspidal representation of the group $\textnormal{SU}(1,1)(\qpp/\qp)$.  For $t\in T_0$, we let $\pi^t$ denote the representation with the same underlying space as $\pi$, with the action of $g\in\textnormal{SU}(1,1)(\qpp/\qp)$ given by first conjugating $g$ by $t$.  Then $\pi^t\cong \pi$.  
\end{coro}

\begin{proof}
 Note first that $t$ normalizes $I_\s(1)$, which implies we have $\pi^{I_\s(1)} = (\pi^t)^{I_\s(1)}$ as vector spaces.  Equation \eqref{act} (along with Proposition \ref{ssingmods}) shows that the actions of the operators $\T_{n_s}, \T_{n_{s'}}$ and $\T_h$ on these two spaces are the same.  By Theorem \ref{gens} we have $\pi^{I_\s(1)} \cong (\pi^t)^{I_\s(1)}$ as $\hh_{\fpb}(G_\s,I_\s(1))$-modules, and Theorem \ref{bij2} now implies that $\pi\cong \pi^t$ as $G_\s$-representations.  
\end{proof}

\vspace{\baselineskip}

\section{Nonsupercuspidal Representations}\label{secnonsc}
In \cite{Ab13}, Abdellatif has classified the smooth, irreducible, nonsupercuspidal representations of connected quasisplit reductive groups of relative rank 1 (or more accurately, the groups of $F$-rational points of these algebraic groups).  In particular, these results apply to the group $G = \textnormal{U}(1,1)(E/F)$.  We recall the results here.

\begin{theo}[\cite{Ab13}, Th\'eor\`eme 1.1]\label{ramla1}
 Let $\chi:T\longrightarrow \fpb^\times$ be a smooth character of $T$, which we inflate to a smooth character of $B$.  
\begin{enumerate}[(1)]
 \item As a $B$-module, the $\fpb$-representation $\textnormal{ind}_B^G(\chi)|_{B}$ is of length 2, with irreducible subquotients given by the character $\chi$ and the representation $V_\chi$, consisting of elements of $\textnormal{ind}_B^G(\chi)$ which take the value $0$ at the identity.
 \item The following are equivalent:
\begin{enumerate}[(a)]
  \item The $B$-module $\textnormal{ind}_B^G(\chi)|_B$ is semisimple;
  \item The $G$-module $\textnormal{ind}_B^G(\chi)$ is reducible;
  \item The $B$-character $\chi$ extends to a character of $G$.
\end{enumerate}
 \item If $\chi$ extends to a character of $G$, then the $G$-module $\textnormal{ind}_B^G(\chi)$ admits as subquotients the representation $\chi$ (as a subrepresentation) and the representation $\chi\otimes\textnormal{St}_G$ (as a quotient).  Here $$\textnormal{St}_G := \textnormal{ind}_B^G(1_B)/1_G$$ denotes the Steinberg representation of $G$, and $1_B$ and $1_G$ denote the trivial characters of $B$ and $G$, respectively.  The short exact sequence 
$$0 \longrightarrow \chi \longrightarrow \textnormal{ind}_B^G(\chi)\longrightarrow \chi\otimes\textnormal{St}_G\longrightarrow 0$$
does not split.  
\end{enumerate}
\end{theo}

This theorem shows that the irreducible nonsupercuspidal representations divide into three families.  The next result demonstrates the lack of isomorphisms between these representations.

\begin{theo}[\cite{Ab13}, Th\'eor\`eme 1.2, Section 3.3]\label{ramla2} 
There do not exist any isomorphisms between representations from distinct families.  If $\chi$ and $\chi'$ are two characters of $B$ which extend to $G$ (resp. do not extend to $G$) and there exists an isomorphism $\chi\otimes\textnormal{St}_G\cong\chi'\otimes\textnormal{St}_G$ (resp. $\textnormal{ind}_B^G(\chi)\cong \textnormal{ind}_B^G(\chi')$), then $\chi = \chi'$.  
\end{theo}

We make Theorem \ref{ramla1} explicit.  For any finite extension $L$ of $F$, we let $\omega$ denote the character of $L^\times$ whose value at a fixed uniformizer $\varpi_L$ is 1, and whose restriction to $\oo_L^\times$ is given by the composition 
\begin{equation}\label{omega}
\omega:\oo_L^\times \stackrel{\mathfrak{r}_L}{\longrightarrow} k_L^\times \stackrel{\iota}{\longrightarrow} \fpb^\times,
\end{equation}
where $\mathfrak{r}_L:\oo_L\longrightarrow k_L$ denotes the reduction modulo the maximal ideal.  For $\lambda\in \fpb^\times$, we denote by $\mu_\lambda:L^\times\longrightarrow \fpb^\times$ the unramified character taking the value $\lambda$ at $\varpi_L$.  In this notation, we have that any smooth character of $E^\times$ (resp. $F^\times$) is of the form $\mu_\lambda\omega^r$ for a unique $\lambda\in\fpb^\times$ and a unique $0\leq r < q^2 - 1$ (resp. $0\leq r < q - 1$).  Likewise, any smooth character of $\textnormal{U}(1)(E/F)$ is of the form $\omega^r$ for a unique $0\leq r < q+1$.

Since $T\cong E^\times$, we will identify the smooth characters of $T$ (and $B$) with those of $E^\times$, by the formula 
$$\mu_\lambda\omega^r\begin{pmatrix}a & b \\ 0 & \ol{a}^{-1}\end{pmatrix} := \mu_\lambda\omega^r(a).$$
Assume now that the character $\mu_\lambda\omega^r$ extends to a character of $G$.  This extension must therefore be trivial on the derived subgroup $G_\s$ and its maximal torus $T_\s$.  In particular, this implies 
\begin{center}
\begin{tabular}{ccccc}
$\lambda$ &  = & $\mu_\lambda\omega^r(\varpi)$ & = & $1$,\\
$a^r$ & = & $\mu_\lambda\omega^r([a])$ & = & $1$,
\end{tabular}
\end{center}
where $a\in \mathbb{F}_q^\times$ is arbitrary.  Hence, we see that if the character $\mu_\lambda\omega^r$ extends to $G$, we must have $\lambda = 1$ and $r = (q - 1)m$ for $0\leq m < q + 1$.  The converse statement is easily verified, and combining Theorems \ref{ramla1} and \ref{ramla2}, we obtain the following theorem.

\begin{theo}\label{nonscclass}
 Let $\pi$ be a smooth, irreducible, nonsupercuspidal representation of the group $\textnormal{U}(1,1)(E/F)$.  Then $\pi$ is isomorphic to one and only one of the following representations:
\begin{itemize}
 \item the smooth $\fpb$-characters $\omega^k\circ\det$, where $0\leq k < q + 1;$
 \item twists of the Steinberg representation $(\omega^k\circ\det)\otimes\textnormal{St}_G$, where $0\leq k < q + 1;$
 \item the principal series representations $\textnormal{ind}_B^G(\mu_\lambda\omega^r)$, where $\lambda\in\fpb^\times$ and $0\leq r < q^2 - 1$ with $(r, \lambda)\neq ((q - 1)m,1)$.
\end{itemize}
\end{theo}

In addition to this classification, we will also need to know how the irreducible representations of $\textnormal{U}(1,1)(E/F)$ behave upon restriction to the derived subgroup $\textnormal{SU}(1,1)(E/F)$.  We record some results in this direction.  

\begin{prop}\label{nonscres}  We have the following isomorphisms:
\begin{enumerate}[(a)]
 \item $\omega^k\circ\det|_{G_\s} \cong 1_{G_\s}$, where $1_{G_\s}$ denotes the trivial character of $G_\s$;
 \item $\textnormal{ind}_B^G(\mu_\lambda\omega^r)|_{G_\s} \cong \textnormal{ind}_{B_\s}^{G_\s}(\mu_\lambda\omega^{r'})$, where $r'$ denotes the unique integer such that $0\leq r' < q - 1$ and $r' \equiv r~ (\textnormal{mod}~ q - 1)$;
 \item $(\omega^k\circ\det)\otimes\textnormal{St}_G|_{G_\s} \cong \textnormal{St}_{G_\s}$, where $\textnormal{St}_{G_\s} := \textnormal{ind}_{B_\s}^{G_\s}(1_{B_\s})/1_{G_\s}$ is the Steinberg representation of $G_\s$.  
\end{enumerate}
\end{prop}

\begin{proof}
 Part (a) follows from the definition of $G_\s$.  For part (b), we note that (by Hilbert's Theorem 90) we have $G = BG_\s = G_\s B$, and use the Mackey decomposition (cf. \cite{Vig96} Chapitre 1, Section 5.5): 
$$\textnormal{ind}_B^G(\mu_\lambda\omega^r)|_{G_\s} = \textnormal{ind}_{B_\s}^{G_\s}(\mu_\lambda\omega^r|_{B_\s}) = \textnormal{ind}_{B_\s}^{G_\s}(\mu_\lambda\omega^{r'}),
$$
where $r'$ is as in the statement of the theorem.  Finally, since the restriction functor $\pi\longmapsto \pi|_{G_\s}$ is exact, we obtain $(\omega^k\circ\det)\otimes\textnormal{St}_{G}|_{G_\s} \cong \textnormal{St}_{G_\s}$ from the definitions of $\textnormal{St}_{G}$ and $\textnormal{St}_{G_\s}$.  This proves (c).  
\end{proof}

\begin{coro}\label{nonscres2}
 If $\pi$ is a smooth irreducible nonsupercuspidal representation of $\textnormal{U}(1,1)(E/F)$, then $\pi$ remains irreducible and nonsupercuspidal upon restriction to $\textnormal{SU}(1,1)(E/F)$.  
\end{coro}

\begin{proof}
 This follows from Theorem \ref{nonscclass} and Proposition \ref{nonscres} above, as well as Proposition 2.7 of \cite{Ab12}.  
\end{proof}

\begin{prop}\label{LLss}
 Let $\pi$ be a smooth irreducible representation of $G$.  Then $\pi$ admits a central character, and the restriction $\pi|_{G_\s}$ is semisimple of length at most $2$ (resp. at most $2^{[F:\mathbb{Q}_2] + 1}$) when $p\neq 2$ (resp. $p = 2$).
\end{prop}

\begin{proof}
We proceed as in Lemma 2.4 of \cite{LL79}.  We first note that, since $K_1$ is a pro-$p$ group, the vector space $\pi^{K_1}$ is nonzero, and has an action of the group $K$.  This action factors through $\Gamma$, and we let $\sigma\subset \pi^{K_1}$ be an irreducible $\Gamma$-subrepresentation (such a representation always exists since $\Gamma$ is finite).  As $\sigma$ is irreducible, the center of $K$ acts on $\sigma$ by a character.  The injection 
$$\sigma\longhookrightarrow \pi|_K$$ 
gives, by Frobenius Reciprocity, a $G$-equivariant surjection $\textnormal{c-ind}_K^G(\sigma)\longtwoheadrightarrow\pi$.  Since $K$ contains the center $Z$, we conclude that $\pi$ admits a central character.

Suppose now that we have a locally profinite group $\GG$, a closed (normal) subgroup $\HH$ of index 2, and an irreducible representation $\pi$ of $\GG$.  Let $\boldsymbol{\theta}\in \GG$ be a representative of the nontrivial coset.  Consider the representation $\pi|_{\HH}$, and suppose it is reducible, with a nonzero proper subrepresentation $\tau$.  Let $V_\tau$ be the underlying vector space of $\tau$.  It then follows that the space $V_\tau + \boldsymbol{\theta}^{-1}.V_\tau$ is nonzero and stable by $\GG$, and therefore must be all of $\pi$.  Likewise, the space $V_\tau \cap \boldsymbol{\theta}^{-1}.V_\tau$ is stable by $\GG$, and hence must be zero.  Thus we see that 
$$\pi|_{\HH} \cong \tau \oplus \tau^{\boldsymbol{\theta}},$$
where $\tau^{\boldsymbol{\theta}}$ denotes the representation with the same underlying space as $\tau$, with the action given by first conjugating an element of $\HH$ by $\boldsymbol{\theta}$.  The same argument shows that $\tau$ must be irreducible as a representation of $\HH$.  

Applying this to the inclusion $ZG_\s \leq \langle ZG_\s, \theta \rangle = G$ when $p \neq 2$ (resp. the chain of inclusions
$$ZG_\s \leq \langle ZG_\s, \theta_1\rangle \leq \langle ZG_\s,\theta_1, \theta_2\rangle \leq \ldots \leq \langle ZG_\s, \theta_1, \ldots, \theta_{[F:\mathbb{Q}_2] + 1}\rangle = G$$
when $p = 2$) shows that the restriction to $ZG_\s$ is semisimple of length at most 2 when $p \neq 2$ (resp. at most $2^{[F:\mathbb{Q}_2] + 1}$ when $p = 2$).  Since $\pi$ admits a central character, the representations appearing in a direct sum decomposition will remain irreducible upon further restricting to $G_\s$.  
\end{proof}

\begin{prop}\label{nonsciffnonsc}
 Let $\pi$ be a smooth irreducible representation of $G$, and let $\tau \subset \pi|_{G_\s}$ be a nonzero irreducible $G_\s$-subrepresentation.  Then $\pi$ is supercuspidal if and only if $\tau$ is supercuspidal.  
\end{prop}

\begin{proof}
 Suppose that $\pi|_{G_\s}$ contains an irreducible nonsupercuspidal representation $\tau$.  If $\tau$ is the trivial character of $G_\s$, then Proposition \ref{LLss} implies $\pi$ is finite-dimensional.  This implies by smoothness and irreducibility that $\pi$ must itself be a character, and hence not supercuspidal.  We may therefore assume that $\tau$ is a quotient of a parabolically induced representation, that is to say, there exists $\chi:B_\s\longrightarrow \fpb^\times$ such that $\tau$ is a quotient of $\textnormal{ind}_{B_\s}^{G_\s}(\chi)$.  Let $\widetilde{\chi}:ZB_\s\longrightarrow \fpb^\times$ denote the character whose restriction to $B_\s$ is $\chi$, and whose restriction to $Z$ is the central character of $\pi$.  Mackey theory now implies that 
$$\textnormal{Hom}_{ZG_\s}(\textnormal{ind}_{ZB_\s}^{ZG_\s}(\widetilde{\chi}), \pi|_{ZG_\s})\neq 0.$$
Using Frobenius Reciprocity and transitivity of induction, we obtain
\begin{eqnarray*}
 \textnormal{Hom}_{ZG_\s}(\textnormal{ind}_{ZB_\s}^{ZG_\s}(\widetilde{\chi}), \pi|_{ZG_\s}) & \cong & \textnormal{Hom}_{G}(\textnormal{c-ind}_{ZG_\s}^G(\textnormal{ind}_{ZB_\s}^{ZG_\s}(\widetilde{\chi})), \pi)\\
 & \cong & \textnormal{Hom}_{G}(\textnormal{ind}_{ZG_\s}^G(\textnormal{ind}_{ZB_\s}^{ZG_\s}(\widetilde{\chi})), \pi)\\
 & \cong & \textnormal{Hom}_{G}(\textnormal{ind}_{ZB_\s}^G(\widetilde{\chi}), \pi)\\
 & \cong & \textnormal{Hom}_{G}(\textnormal{ind}_{B}^G(\textnormal{ind}_{ZB_\s}^{B}(\widetilde{\chi})), \pi).
\end{eqnarray*}

Since every irreducible finite-dimensional representation of $B$ is a character, we have that $\textnormal{ind}_{ZB_\s}^B(\widetilde{\chi})$ is of length $n$, where $n = 2$ if $p\neq 2$ and $n = 2^{[F:\mathbb{Q}_2] + 1}$ if $p = 2$.  Suppose
$$0 = M_0 \subsetneq M_1\subsetneq \ldots \subsetneq M_{n} = \textnormal{ind}_{ZB_\s}^B(\widetilde{\chi})$$
is a composition series for $\textnormal{ind}_{ZB_\s}^B(\widetilde{\chi})$, and set $\chi_i := M_i/M_{i - 1}$ for $i = 1,\ldots, n$.  We claim that there exists $\chi_i$ such that $\textnormal{Hom}_G(\textnormal{ind}_B^G(\chi_i),\pi)\neq 0$.  Assume the contrary.  Since parabolic induction is exact (\cite{Vig13}, Proposition 4.3), the exact sequence
$$0 \longrightarrow \chi_i \longrightarrow M_n/M_{i - 1} \longrightarrow M_n/M_i\longrightarrow 0$$
of $B$-representations yields an exact sequence
$$0 \longrightarrow \textnormal{ind}_B^G(\chi_i) \longrightarrow \textnormal{ind}_B^G(M_n/M_{i - 1}) \longrightarrow \textnormal{ind}_B^G(M_n/M_i)\longrightarrow 0$$
of $G$-representations for every $i = 1,\ldots, n$ (noting that $U$ acts trivially on $M_n$).  Applying $\textnormal{Hom}_G(-,\pi)$ gives an exact sequence
$$0 \longrightarrow \textnormal{Hom}_G(\textnormal{ind}_B^G(M_n/M_{i}),\pi)\longrightarrow \textnormal{Hom}_G(\textnormal{ind}_B^G(M_n/M_{i - 1}),\pi)\longrightarrow \textnormal{Hom}_G(\textnormal{ind}_B^G(\chi_i),\pi) = 0,$$
where the last equality holds by assumption.  This gives
\begin{eqnarray*}
\textnormal{Hom}_G(\textnormal{ind}_B^G(M_n),\pi) & \cong & \textnormal{Hom}_G(\textnormal{ind}_B^G(M_n/M_1),\pi)\\
 & \cong & \ldots \quad\cong \textnormal{Hom}_G(\textnormal{ind}_B^G(M_n/M_{n - 1}),\pi) = 0,
 \end{eqnarray*}
a contradiction.  From this we see that $\pi$ is a quotient of a representation parabolically induced from a character, and is therefore not supercuspidal.  

The reverse implication follows from Corollary \ref{nonscres2}.
\end{proof}

\vspace{\baselineskip}

\section{Supercuspidal Representations}

We suppose from this point onwards that $F = \qp, E = \qpp$, and $\varpi = p$.  We let $\zp$ and $\zpp$ denote the rings of integers of $\qp$ and $\qpp$, respectively.  

The supercuspidal representations of $\textnormal{GL}_2(\qp)$ and $\textnormal{SL}_2(\qp)$ have been classified by Breuil and Abdellatif, respectively (cf. \cite{Br03}, \cite{Ab12}).  We review their results here.

\vspace{\baselineskip}
\subsection{The Group $\textnormal{GL}_2(\qp)$}\label{gl2sc} Let $0\leq r \leq p - 1$ be an integer.  Denote by $\sigma_r = \textnormal{Sym}^r(\fpb^2)$ the $\fpb$-vector space of homogeneous polynomials in two variables of degree $r$, with an action of $\textnormal{GL}_2(\zp)$ given by
\begin{eqnarray*}
 \begin{pmatrix}a & b \\ c & d\end{pmatrix}.x^{r - i}y^i & = & (\mathfrak{r}_{\qp}(a)x + \mathfrak{r}_{\qp}(c)y)^{r - i}(\mathfrak{r}_{\qp}(b)x + \mathfrak{r}_{\qp}(d)y)^i,
\end{eqnarray*}
where $\mathfrak{r}_{\qp}:\zp\longrightarrow \mathbb{F}_p$ is the reduction map defined in Section \ref{secnonsc}.  We denote by $\qp^\times$ the center of $\textnormal{GL}_2(\qp)$, and extend the above action to $\qp^\times\textnormal{GL}_2(\zp)$ by letting $p\cdot \textnormal{id}$ act trivially.  Proposition 8 of \cite{BL94} shows that the algebra of $\textnormal{GL}_2(\qp)$-equivariant endomorphisms of the compactly induced representation $\textnormal{c-ind}_{\qp^\times\textnormal{GL}_2(\zp)}^{\textnormal{GL}_2(\qp)}(\sigma_r)$ is isomorphic to a polynomial algebra over $\fpb$ in one variable, generated by an endomorphism denoted $T_r$.  For a smooth character $\chi$ of $\qp^\times$, we denote by $\pi(r,0,\chi)$ the representation of $\textnormal{GL}_2(\qp)$ afforded by the cokernel of the map $T_r$, twisted by $\chi$~: 
$$\pi(r,0,\chi) := (\chi\circ\det)\otimes\frac{\textnormal{c-ind}_{\qp^\times\textnormal{GL}_2(\zp)}^{\textnormal{GL}_2(\qp)}(\sigma_r)}{(T_r)}~.$$

The necessary properties of the representations $\pi(r,0,\chi)$ are summarized in the following theorem, proved by Barthel--Livn\'{e} and Breuil.  

\begin{theo}\label{BLandBr}
 In the following, $r$ denotes an integer $0\leq r \leq p - 1$ and $\chi:\qp^\times\longrightarrow\fpb^\times$ a smooth character.
\begin{enumerate}[(a)]
 \item Every smooth irreducible supercuspidal representation of $\textnormal{GL}_2(\qp)$ is of the form $\pi(r,0,\chi)$.  
 \item The only isomorphisms among the representations $\pi(r,0,\chi)$ are the following~:
\begin{eqnarray*}
 \pi(r,0,\chi) & \cong & \pi(r,0,\chi\mu_{-1})\\
 \pi(r,0,\chi) & \cong & \pi(p - 1 - r,0,\chi\omega^r)\\
 \pi(r,0,\chi) & \cong & \pi(p - 1 - r,0,\chi\mu_{-1}\omega^r).
\end{eqnarray*}
 \item Let $Iw(1)$ denote the standard upper pro-$p$-Iwahori subgroup of $\textnormal{GL}_2(\qp)$.  We have $$\pi(r,0,1)^{Iw(1)} = \fpb\ol{[\textnormal{id}, x^r]} \oplus \fpb\ol{[\beta,x^r]},$$ where, for $g\in \textnormal{GL}_2(\qp)$ and $v\in\sigma_r$, $\ol{[g,v]}$ denotes the image in $\pi(r,0,1)$ of the element $[g,v]$ of $\textnormal{c-ind}_{\qp^\times\textnormal{GL}_2(\zp)}^{\textnormal{GL}_2(\qp)}(\sigma_r)$ with support $\qp^\times\textnormal{GL}_2(\zp)g^{-1}$ and value $v$ at $g$.
\end{enumerate}
\end{theo}

\begin{proof}
 This follows from Theorems 33, 34 and Corollary 36 of \cite{BL94}, and Th\'eor\`eme 3.2.4 and Corollaires 4.1.1, 4.1.4, and 4.1.5 of \cite{Br03}.  We remark that the hypothesis of having a central character made in \cite{BL94}, \cite{BL95}, and \cite{Br03} may be omitted by \cite{Be11}.
\end{proof}

\vspace{\baselineskip}
\subsection{The Group $\textnormal{SL}_2(\qp)$} Consider now the group $\textnormal{SL}_2(\qp)$.  Th\'eor\`eme 3.36 of \cite{Ab12} implies that for any smooth irreducible representation $\sigma$ of $\textnormal{SL}_2(\qp)$, there exists a smooth irreducible representation $\Sigma$ of $\textnormal{GL}_2(\qp)$ such that $\sigma$ is a Jordan--H\"older factor of $\Sigma$.  Moreover, Corollaires 3.38 and 3.41 (\emph{loc. cit.}) imply that in order to classify supercuspidal representations of $\textnormal{SL}_2(\qp)$, it suffices to compute the restriction of the representations $\pi(r,0,1)$.  Let $\pi_{r,\infty}$ denote the $\textnormal{SL}_2(\qp)$-subrepresentation of $\pi(r,0,1)|_{\textnormal{SL}_2(\qp)}$ generated by $\ol{[\textnormal{id}, x^r]}$, and let $\pi_{r,0}$ denote the $\textnormal{SL}_2(\qp)$-subrepresentation of $\pi(r,0,1)|_{\textnormal{SL}_2(\qp)}$ generated by $\ol{[\beta,x^r]}$.  

\begin{theo}\label{Ramla}
 In the following, $r$ denotes an integer $0\leq r \leq p - 1$.  
\begin{enumerate}[(a)]
 \item We have $\pi(r,0,1)|_{\textnormal{SL}_2(\qp)} \cong \pi_{r,\infty} \oplus \pi_{r,0}$.
 \item The representations $\pi_{r,0}$ and $\pi_{r,\infty}$ are smooth, irreducible, admissible, and supercuspidal.
 \item Conversely, any smooth, irreducible, and supercuspidal representation of $\textnormal{SL}_2(\qp)$ is isomorphic to one of the form $\pi_{r,0}$ or $\pi_{r,\infty}$.
 \item The only isomorphisms among the representations $\pi_{r,0}$ and $\pi_{r,\infty}$ are the following:
\begin{eqnarray*}
 \pi_{r,\infty} & \cong & \pi_{p - 1 - r,0}
\end{eqnarray*}
 \item Let $Iw_\s(1)$ denote the standard upper pro-$p$-Iwahori subgroup of $\textnormal{SL}_2(\qp)$.  We have
\begin{eqnarray*}
 \pi_{r,\infty}^{Iw_\s(1)} & = & \fpb\ol{[\textnormal{id}, x^r]},\\
 \pi_{r,0}^{Iw_\s(1)} & = & \fpb\ol{[\beta,x^r]}.
\end{eqnarray*}
\end{enumerate}
\end{theo}

\begin{proof}
 This follows from Propositions 4.5 and 4.7 and Corollaires 4.8 and 4.9 of \cite{Ab12}, and the comments following Corollaire 4.9.
\end{proof}

\begin{defi}\label{defofsl2ssing}
  We let $\pi_r$ denote the representation $\pi_{r,\infty}$.  In light of Theorem \ref{Ramla}, the representations $\pi_r$ with $0\leq r \leq p - 1$ are a full set of representatives for the isomorphism classes of supercuspidal representations of $\textnormal{SL}_2(\qp)$ (cf. \cite{Ab12}, Th\'eor\`eme 4.12).  
\end{defi}

We shall henceforth view the representations $\pi_r$ as representations of $\textnormal{SU}(1,1)(\qpp/\qp)$ via the isomorphism $\textnormal{SU}(1,1)(\qpp/\qp)\cong\textnormal{SL}_2(\qp)$.  

\vspace{\baselineskip}
\subsection{The Group $\textnormal{U}(1,1)(\qpp/\qp)$}  We now proceed to examine the supercuspidal representions of the group $\textnormal{U}(1,1)(\qpp/\qp)$.

\begin{defi}
Let $\textnormal{U}(1)_1 := \textnormal{U}(1)(\qpp/\qp)\cap 1 + p\zpp$.  We define $G_0$ to be the subgroup of $G$ generated by $G_\s$ and the central subgroup
$$\left\{\begin{pmatrix}a & 0 \\ 0 & a\end{pmatrix}: a\in \textnormal{U}(1)_1\right\}.$$
\end{defi}

The group $G_0$ fits into an exact sequence
\begin{equation}\label{gzeroseq}
1 \longrightarrow G_\s \longrightarrow G_0 \stackrel{\det}{\longrightarrow} \textnormal{U}(1)_1^2 \longrightarrow 1,
\end{equation}
and we have $G/G_0\cong \textnormal{U}(1)/\textnormal{U}(1)_1^2$.

Assume $p \neq 2$.  Proposition 6(b), Chapitre IV and Lemme 2, Chapitre V of \cite{Se68} imply that the map $x\longmapsto x^2$ is an automorphism of $\textnormal{U}(1)_1$, and therefore the map
$$x^2\longmapsto \begin{pmatrix}x & 0 \\ 0 & x\end{pmatrix}$$
gives a well-defined section to the determinant map in the short exact sequence \eqref{gzeroseq} above.  Thus, we obtain
$$G_0 \cong G_\s\times \textnormal{U}(1)_1^2 \cong G_\s\times \textnormal{U}(1)_1.$$
Since $\textnormal{U}(1)_1$ is a pro-$p$ group, Lemma 3 of \cite{BL94} implies that the set of smooth, irreducible mod-$p$ representations of $G_0$ and $G_\s$ are in canonical bijection.  

Assume now that $p = 2$.  The map $x\longmapsto x^2$ is no longer an automorphism of $\textnormal{U}(1)_1$ (note that $\ker(x\longmapsto x^2) = \{\pm 1\}\subset \textnormal{U}(1)_1$).  Pushing out the exact sequence \eqref{gzeroseq} by the map $G_\s\longrightarrow G_\s/\{\pm \textnormal{id}\}$ gives an exact sequence
$$1 \longrightarrow G_\s/\{\pm \textnormal{id}\} \longrightarrow G_0/\{\pm \textnormal{id}\} \stackrel{\det}{\longrightarrow} \textnormal{U}(1)_1^2 \longrightarrow 1.$$
The map sending $x^2$ to the class of $\left(\begin{smallmatrix}x & 0 \\ 0 & x\end{smallmatrix}\right)$ gives a well-defined section to the short exact sequence above, and therefore we obtain
$$G_0/\{\pm \textnormal{id}\}\cong G_\s/\{\pm \textnormal{id}\} \times \textnormal{U}(1)_1^2.$$
Once again, the set of smooth, irreducible mod-2 representations of $G_0/\{\pm \textnormal{id}\}$ and $G_\s/\{\pm \textnormal{id}\}$ are in canonical bijection.  The argument of Proposition \ref{LLss} shows that irreducible representations of $G_0$ and $G_\s$ admit central characters, which take the value 1 at $\pm \textnormal{id}$.  This shows that the smooth, irreducible mod-2 representations of $G_0$ and $G_\s$ are in canonical bijection.


We again assume $p$ is arbitrary.  Hilbert's Theorem 90 implies that $T_0G_0 = G$, so we may and do choose coset representatives $\{\delta_i\}_{i\in \textnormal{U}(1)/\textnormal{U}(1)_1^2}$ for $G/G_0$ such that $\delta_i\in T_0$.  Now let $0\leq r \leq p - 1$, and let $\pi_r$ be a smooth irreducible supercuspidal representation of $G_\s$, inflated to $G_0$.  For $\delta_i$ as above, Corollary \ref{twistisom} implies $\pi_r\cong \pi_r^{\delta_i}$ as $G_\s$-representations (and consequently as $G_0$-representations).  Therefore, we may lift $\pi_r$ to a projective representation of $G$.  Since
\begin{eqnarray*}
\textnormal{H}^2(G/G_0,\fpb^\times) & = & \textnormal{H}^2(\textnormal{U}(1)/\textnormal{U}(1)_1^2,\fpb^\times)\\
& = & \begin{cases}\textnormal{H}^2(\mathbb{Z}/(p + 1)\mathbb{Z},\fpb^\times) & \textnormal{if}~p\neq 2,\\ \textnormal{H}^2(\mathbb{Z}/3\mathbb{Z} \oplus (\mathbb{Z}/2\mathbb{Z})^{\oplus 2},\overline{\mathbb{F}}_2^\times) & \textnormal{if}~p = 2,\end{cases}\\
& = & 0,
\end{eqnarray*}
this representation lifts to a genuine representation $\widetilde{\pi}_r$ of $G$.  For more details, see \cite{JMY06}.

It remains to determine the action of $G$ on the lift $\widetilde{\pi}_r$ of $\pi_r$.  Consider the homomorphism $h_s:\qpp^\times \longrightarrow T$ defined by
$$h_s(a) = \begin{pmatrix}a & 0 \\0 & \ol{a}^{-1}\end{pmatrix}.$$  
Since $0\neq \widetilde{\pi}_r^{I(1)}\subset \widetilde{\pi}_r^{I_\s(1)}$, Theorem \ref{Ramla}(e) implies $\widetilde{\pi}_r^{I(1)} = \fpb v_r$, where $v_r = \ol{[\textnormal{id},x^r]}$.  As the elements $h_s([a])$ for $a\in \mathbb{F}_{p^2}^\times$ normalize $I(1)$, we have $$h_s([a]).v_r = a^mv_r,$$ for some $0 \leq m < p^2 - 1$.  Since $a^{p + 1}\in \mathbb{F}_p^\times$ for $a\in \mathbb{F}_{p^2}^\times$, we have $h_s([a]^{p + 1})\in T_{0,\s}$ and $$a^{(p + 1)r}v_r = h_s([a]^{p+1}).v_r = a^{(p + 1)m}v_r,$$ by the action of $I_\s$ on $\ol{[\textnormal{id},x^r]}\in\pi_r^{I_\s(1)}$.  Thus, we must have $m = r + (1 - p)k$ for some $k\in \mathbb{Z}$.  This leads to the following definition.

\begin{defi}\label{defofsc}
 Let $0\leq r \leq p - 1$ and $0\leq k < p+1$.  We define the representation $(\omega^k\circ\det)\otimes\bpi_r$ of $G = \textnormal{U}(1,1)(\qpp/\qp)$ by the following conditions:
\begin{itemize}
 \item $(\omega^k\circ\det)\otimes\bpi_r|_{G_\s} = \pi_r$;
 \item $((\omega^k\circ\det)\otimes\bpi_r)^{I(1)} = \pi_r^{I_\s(1)} = \fpb v_r$ as vector spaces;
 \item $h_s([a]).v_r = a^{r + (1 - p)k}v_r$ for $a\in \mathbb{F}_{p^2}^\times$.
\end{itemize}
\end{defi}

\vspace{\baselineskip}

The last point defines the character that gives the action of $I$ on $((\omega^k\circ\det)\otimes\bpi_r)^{I(1)}$, and the preceding discussion ensures that the vector spaces $(\omega^k\circ\det)\otimes\bpi_r$ are bona fide representations of $G$.  We collect their properties in the following proposition.

\begin{prop}\label{propsc}
 Let $0\leq r \leq p - 1$ and $0\leq k < p + 1$.  
\begin{enumerate}[(a)]
 \item The representations $(\omega^k\circ\det)\otimes\bpi_r$ are smooth, irreducible, admissible, and supercuspidal representations of $G$.
 \item The representations $(\omega^k\circ\det)\otimes\bpi_r$ are pairwise nonisomorphic.
\end{enumerate}
\end{prop}

\begin{proof}
 (a) The claim about irreducibility follows from the fact that the restriction of the representation $(\omega^k\circ\det)\otimes\bpi_r$ to $G_\s$ is irreducible.  The space of $I(1)$-invariants is one-dimensional, and therefore $(\omega^k\circ\det)\otimes\bpi_r$ is admissible.  Proposition \ref{nonsciffnonsc} implies that the representations are supercuspidal.  

 (b) Suppose that $$\varphi:(\omega^k\circ\det)\otimes\bpi_r\longrightarrow (\omega^{k'}\circ\det)\otimes\bpi_{r'}$$ is a $G$-equivariant isomorphism.  The map $\varphi$ then defines a $G_\s$-equivariant isomorphism, so we must have $r = r'$ by Theorem \ref{Ramla}.  The last point of Definition \ref{defofsc} now shows that we must have $k \equiv k'~(\textnormal{mod}~ p + 1)$, which implies $k = k'$ (since $0\leq k,k' < p + 1$).  
\end{proof}

\begin{theo}\label{scclass}
 Let $\pi$ be a smooth irreducible supercuspidal representation of the group $G = \textnormal{U}(1,1)(\qpp/\qp)$.  Then $\pi$ is isomorphic to a unique representation of the form $(\omega^k\circ\det)\otimes\bpi_r$ with $0\leq r \leq p - 1$ and $0\leq k < p + 1$.  
\end{theo}

\begin{proof}
 Let $\pi$ be a smooth irreducible supercuspidal representation.  The proofs of Propositions \ref{LLss} and \ref{nonsciffnonsc} and Corollary \ref{twistisom} imply that $\pi|_{G_\s}$ must be of the form 
$$\pi|_{G_\s}\cong \bigoplus_{t\in \mathcal{J}}\pi_r^t \cong \pi_r^{\oplus |\mathcal{J}|},$$ 
for some $0\leq r \leq p - 1$, and $\mathcal{J}$ a finite subset of $T_0$.  

We claim that $|\mathcal{J}| = 1$.  To see this, note first that $\pi|_{G_\s}$ is an object of $\mathfrak{Rep}_{\fpb}^{I_\s(1)}(G_\s)$, the full subcategory of $\mathfrak{Rep}_{\fpb}(G_\s)$ consisting of representations generated by their $I_\s(1)$-invariants.  The functor of $I_\s(1)$-invariants restricted to this subcategory is faithful (this follows from \cite{BL94}, Lemma 3(1)), and we obtain an injection
$$\textnormal{End}_{G_\s}(\pi|_{G_\s})\longhookrightarrow \textnormal{End}_{\fpb}(\pi^{I_\s(1)}),$$
given by restricting endomorphisms to $\pi^{I_\s(1)}$.  Counting dimensions shows that this map is bijective.  

Now fix $t_0\in\mathcal{J}$ and a nonzero $v_0\in\pi_r^{I_\s(1)}$, the latter space indexed by $t_0$ in the direct sum above.  Let $v\in \pi^{I_\s(1)}$ be a nonzero eigenvector for $T_0$.  There exists an $\fpb$-linear automorphism of $\pi^{I_\s(1)}$ taking $v$ to $v_0$, and by the above remarks, we obtain a $G_\s$-equivariant automorphism $\varphi$ of $\pi|_{G_\s}$ taking $v$ to $v_0$.  Consider the subspace $\langle G.v \rangle_{\fpb} = \langle G_\s T_0.v\rangle_{\fpb} = \langle G_\s.v\rangle_{\fpb}$.  Since it is stable by $G$, it must be all of $\pi$.  On the other hand, the map $\varphi|_{\langle G_\s.v\rangle_{\fpb}}$ gives a $G_\s$-equivariant isomorphism
$$\pi|_{G_\s} = \langle G_\s.v\rangle_{\fpb}|_{G_\s}\stackrel{\sim}{\longrightarrow}\langle G_\s.v_0\rangle_{\fpb}|_{G_\s}\cong \pi_r.$$
This gives the claim.  

 We may therefore assume that $\pi|_{G_\s}\cong \pi_r$.  The discussion preceding Definition \ref{defofsc} then shows that there exists an integer $k$ such that $\pi\cong (\omega^k\circ\det)\otimes\bpi_r$.  
\end{proof}

\begin{coro}\label{u11class}
 Let $\pi$ be a smooth irreducible representation of $G = \textnormal{U}(1,1)(\qpp/\qp)$.  Then $\pi$ admits a central character and is admissible.  Moreover, $\pi$ is isomorphic to one and only one of the following representations:
\begin{itemize}
 \item the smooth $\fpb$-characters $\omega^k\circ\det$, where $0\leq k < p + 1;$
 \item twists of the Steinberg representation $(\omega^k\circ\det)\otimes\textnormal{St}_G$, where $0\leq k < p + 1;$
 \item the principal series representations $\textnormal{ind}_B^G(\mu_\lambda\omega^r)$, where $\lambda\in\fpb^\times$ and $0\leq r < p^2 - 1$ with $(r, \lambda)\neq ((p - 1)m, 1);$
\item the supercuspidal representations $(\omega^k\circ\det)\otimes\bpi_r$, where $0\leq r \leq p - 1$ and $0 \leq k < p + 1$.
\end{itemize}
\end{coro}

\begin{proof}
 If $\pi$ is not supercuspidal, then the result follows from Theorem \ref{nonscclass}, and if $\pi$ is supercuspidal it follows from Proposition \ref{propsc} and Theorem \ref{scclass}.  It only remains to prove that no supercuspidal representation is isomorphic to a nonsupercuspidal representation.  Assume this is the case; we then obtain a $G_\s$-equivariant isomorphism between a supercuspidal representation and a nonsupercuspidal representation, contradicting Corollaire 3.19 of \cite{Ab12}.  
\end{proof}

\vspace{\baselineskip}
\subsection{$L$-packets}
We define the \emph{general unitary group} $\textnormal{GU}(1,1)(\qpp/\qp)$ by $$\left\{g\in \textrm{GL}_2(\qpp): g^*\begin{pmatrix}0 & 1 \\ 1 & 0\end{pmatrix}g = \kappa \begin{pmatrix}0 & 1 \\ 1 & 0\end{pmatrix}~\textnormal{for some}~ \kappa\in\qp^\times\right\}.$$  The association $g\longmapsto \kappa$ is in fact a character, and induces a surjective homomorphism $\textnormal{sim}:\textnormal{GU}(1,1)(\qpp/\qp)\longrightarrow \qp^\times$.  We obtain a short exact sequence of groups 
$$1\longrightarrow \textnormal{U}(1,1)(\qpp/\qp)\longrightarrow \textnormal{GU}(1,1)(\qpp/\qp)\stackrel{\textnormal{sim}}{\longrightarrow}\qp^\times \longrightarrow 1$$ 
which splits, and we have 
\begin{equation}\label{gu(1,1)}
\textnormal{GU}(1,1)(\qpp/\qp) = \textnormal{U}(1,1)(\qpp/\qp)\rtimes\left\{\begin{pmatrix}1 & 0 \\ 0 & a\end{pmatrix}:~ a\in\qp^\times\right\}.
\end{equation}
As $G = \textnormal{U}(1,1)(\qpp/\qp)$ is normal in $\textnormal{GU}(1,1)(\qpp/\qp)$, the latter group acts on $G$ by conjugation, and consequently acts on representations of $G$.  The following definition is adapted from the complex case (see Section 11.1 of \cite{Ro90}).

\begin{defi}
An \emph{$L$-packet of semisimple representations on $G = \textnormal{U}(1,1)(\qpp/\qp)$} is a $\textnormal{GU}(1,1)(\qpp/\qp)$-orbit of smooth semisimple representations of $G$.  An $L$-packet is called \emph{supercuspidal} if it consists entirely of irreducible supercuspidal representations.
\end{defi}

\begin{prop}\label{Lpackets}
 Let $\Pi$ be an $L$-packet of smooth irreducible representations on $G = \textnormal{U}(1,1)(\qpp/\qp)$.  Then $\Pi$ has cardinality 1 if and only if it contains an irreducible nonsupercuspidal representation.  If $\Pi$ is a supercuspidal $L$-packet, then it is of the form $$\Pi = \{\cusp{k}{r},~ \cusp{k + r + 1}{p - 1 - r}\},$$ for some $0\leq r \leq p - 1,~ 0\leq k < p + 1$.  
\end{prop}

\begin{proof}
We begin with the following general fact.  Let $f:G\longrightarrow G$ be a continuous automorphism.  Given an irreducible representation $\pi$, we let $\pi^f$ denote the representation with the same underlying vector space as $\pi$, and the action of $g\in G$ given by first applying $f$ to $g$.  One easily verifies that
$$\textnormal{ind}_B^G(\chi)^f \cong \textnormal{ind}_{f^{-1}(B)}^G(\chi\circ f).$$

 Now let 
$$t = \begin{pmatrix}1 & 0 \\ 0 & a \end{pmatrix}$$ 
with $a\in \qp^\times$, and let $f:G\longrightarrow G$ denote the automorphism $g\longmapsto tgt^{-1}$ given by conjugation by $t$.  Given an irreducible representation $\pi$, we denote $\pi^f$ by $\pi^t$.  

If $\pi = \omega^k\circ\det$ is a character of $G$, it is clear that $(\omega^k\circ\det)^t = \omega^k\circ\det$.  Likewise, if $\pi = \textnormal{ind}_B^G(\mu_\lambda\omega^r)$, then the above fact shows that
$$\pi^t = \textnormal{ind}_B^G(\mu_\lambda\omega^r)^t\cong \textnormal{ind}_{t^{-1}Bt}^G((\mu_\lambda\omega^r)^t) \cong \textnormal{ind}_B^G(\mu_\lambda\omega^r) = \pi.$$
Since the functor $\pi\longmapsto \pi^t$ is exact in the category $\mathfrak{Rep}_{\fpb}(G)$, we conclude that $((\omega^k\circ\det)\otimes\textnormal{St}_G)^t \cong (\omega^k\circ\det)\otimes\textnormal{St}_G$.  Hence, if $\Pi$ is an $L$-packet containing a nonsupersingular representation, then $\Pi$ has size 1.  

 Assume now that $\Pi$ contains a supersingular representation $\pi = \cusp{k}{r}$, and let 
$$t = \begin{pmatrix}1 & 0 \\ 0 & a\end{pmatrix}$$ 
with $a\in \zp^\times$.  Proceeding as in the proof of Corollary \ref{twistisom}, we see that $\pi^t|_{G_\s}\cong \pi|_{G_\s}$, and the action of $T$ on $\cusp{k}{r}$ shows that $\pi^t \cong \pi$.

 To conclude the proof, we must compute $\pi^{\beta s}$, where 
$$\beta s = \begin{pmatrix}1 & 0 \\ 0 & p\end{pmatrix};$$
since $s\in\textnormal{U}(1,1)(\qpp/\qp)$, we have $\pi^s\cong \pi$ and it suffices to determine $\pi^\beta$.  Corollaire 4.6 of \cite{Ab12} implies 
$$\pi^\beta|_{G_\s}\cong \pi_r^\beta \cong \pi_{p - 1 - r},$$ 
and Theorem \ref{scclass} gives $\pi^\beta\cong \cusp{k'}{p - 1 - r}$ for some $k'$.  As the element $\beta$ normalizes $I(1)$, we have 
$$\fpb v_r = \pi^{I(1)} = (\pi^\beta)^{I(1)}$$ 
as vector spaces.  For $a\in \mathbb{F}_{p^2}^\times$, the action of $h_s([a])$ on $(\pi^\beta)^{I(1)}$ is given by
$$h_s([a]).v_r = a^{-pr - pk + k}v_r,$$ 
while the action of $h_s([a])$ on $(\cusp{k'}{p - 1 - r})^{I(1)}$ is given by
$$h_s([a]).v_{p - 1 - r} = a^{p - 1 - r + (1 - p)k'}v_{p - 1 - r}.$$  
These actions coincide, and therefore $(1 - p)k' - (1 - p) - r \equiv (1 - p)k - pr~(\textnormal{mod}~ p^2 - 1)$, which shows $k' \equiv r + k + 1~(\textnormal{mod}~ p + 1)$.  Hence
$$\Pi = \{\cusp{k}{r},~ \cusp{k + r + 1}{p - 1 - r}\},$$ 
which concludes the proof.
\end{proof}

\vspace{\baselineskip}

\section{Galois Groups and Representations}

In this section we recall the definitions associated to Galois representations attached to unitary groups.  

\vspace{\baselineskip}
\subsection{Galois Groups}
Let $\gal_{\qp} := \textrm{Gal}(\ol{\mathbb{Q}}_p/\qp)$ denote the absolute Galois group of $\qp$, and $\ii_{\qp}$ the inertia subgroup.  Given a finite extension $F$ of $\qp$ contained in $\ol{\mathbb{Q}}_p$, we define $\gal_F := \textrm{Gal}(\ol{\mathbb{Q}}_p/F)$.  For $n\geq 1$, we let $\mathbb{Q}_{p^n}$ denote the unique unramified extension of $\qp$ of degree $n$ contained in $\ol{\mathbb{Q}}_p$, with canonical uniformizer $p$, and let 
\begin{equation}\label{rec}
\iota_n:\mathbb{Q}_{p^n}^\times\longrightarrow \gal_{\mathbb{Q}_{p^n}}^{ab}
\end{equation}
denote the reciprocity map of local class field theory, normalized so that uniformizers correspond to geometric Frobenius elements.  We shall denote by $\textnormal{Fr}_p$ a fixed element of $\gal_{\qp}$ whose image in $\gal_{\qp}^{ab}$ is equal to $\iota_1(p^{-1})$.  

Using the injections $\iota_n$, we will identify the smooth $\fpb$-characters of $\mathbb{Q}_{p^n}^\times$ and $\gal_{\mathbb{Q}_{p^n}}$ in the following way.  We fix a compatible system $\{\sqrt[p^n-1]{p}\}_{n\geq 1}$ of $(p^n-1)^{\textnormal{th}}$ roots of $p$, and let $\omega_n:\ii_{\qp}\longrightarrow \fpb^\times$ denote the character given by 
\begin{equation}\label{omegan}
\omega_n: h\longmapsto \iota\circ\mathfrak{r}_{\ol{\mathbb{Q}}_p}\left(\frac{h.\sqrt[p^n-1]{p}}{\sqrt[p^n-1]{p}}\right),
\end{equation}
where $h\in \ii_{\qp}$ and $\mathfrak{r}_{\ol{\mathbb{Q}}_p}:\ol{\mathbb{Z}}_p\longrightarrow k_{\ol{\mathbb{Q}}_p}$ denotes the reduction modulo the maximal ideal.  Lemma 2.5 of \cite{Br07} shows that the character $\omega_n$ extends to a character of $\gal_{\mathbb{Q}_{p^n}}$; we continue to denote by $\omega_n$ the extension which sends the element $\textnormal{Fr}_p^n$ to 1.  

\begin{lemm}\label{firstgal}
 For $n\geq 1$, we have $\omega_n\circ\iota_n = \omega$, where $\omega$ is the character defined in equation \eqref{omega}, which sends $p$ to 1.  
\end{lemm}
\begin{proof}
Denote by $\mathbb{Z}_{p^n}$ the ring of integers of $\mathbb{Q}_{p^n}$, and suppose $u\in\mathbb{Z}_{p^n}^\times$.  Propositions 6 and 8 of Chapitre XIV, \cite{Se68} imply
\begin{eqnarray*}
\frac{\iota_n(u).\sqrt[p^n - 1]{p}}{\sqrt[p^n - 1]{p}} & = & \frac{(u^{-1},*/\mathbb{Q}_{p^n}).\sqrt[p^n - 1]{p}}{\sqrt[p^n - 1]{p}}\\
 & = & (p,u^{-1})_{\nu_n}\\
 & = & \left[\mathfrak{r}_{\ol{\mathbb{Q}}_p}\left((-1)^{\nu_n(p)\nu_n(u^{-1})}\frac{p^{\nu_n(u^{-1})}}{u^{-\nu_n(p)}}\right)\right]\\
 & = & \left[\mathfrak{r}_{\ol{\mathbb{Q}}_p}(u)\right].
\end{eqnarray*}
Here, $(~\cdot~,*/\mathbb{Q}_{p^n}):\mathbb{Q}_{p^n}^\times\longrightarrow \gal_{\mathbb{Q}_{p^n}}^{ab}$ denotes the norm residue symbol of the field $\mathbb{Q}_{p^n}$ (Chapitre XIII, \emph{loc. cit.}), $\nu_n:\mathbb{Q}_{p^n}^\times\longrightarrow \mathbb{Z}$ is the normalized valuation on $\mathbb{Q}_{p^n}$, and $(~\cdot~,~ \cdot~)_{\nu_n}:\mathbb{Q}_{p^n}^\times\times\mathbb{Q}_{p^n}^\times\longrightarrow \boldsymbol{\mu}_{p^n - 1}(\mathbb{Q}_{p^n}^\times)$ denotes the Hilbert symbol (Chapitre XIV, \emph{loc. cit.}).  Applying $\iota\circ\mathfrak{r}_{\ol{\mathbb{Q}}_p}$ to both sides shows that $\omega_n\circ\iota_n(u) = \omega(u)$.

The functorial properties of the reciprocity maps $\iota_n$ imply that we have equalities
$$\omega_n\circ\iota_n(p^{-1}) = \omega_n\circ\textnormal{ver}\circ\iota_1(p^{-1}) = \omega_n(\textnormal{Fr}_p^n) = 1 = \omega(p^{-1}),$$
where $\textnormal{ver}:\gal_{\qp}^{ab}\longrightarrow\gal_{\mathbb{Q}_{p^n}}^{ab}$ denotes the transfer map.  The result now follows.  
\end{proof}

\begin{lemm}\label{brlemma}
For $h\in\ii_{\qp}$ and $n\geq 1$, we have $$\omega_n(\textnormal{Fr}_ph\textnormal{Fr}_p^{-1}) = \omega_n(h)^p.$$ 
\end{lemm}
\begin{proof}
 See the proof of Lemma 2.5 in \cite{Br07}.
\end{proof}

For $\lambda\in\fpb^\times$ and $n\geq 1$, we let $\mu_{n,\lambda}:\gal_{\mathbb{Q}_{p^n}}\longrightarrow \fpb^\times$ denote the unramified character which is trivial on $\ii_{\mathbb{Q}_{p}}$ and sends $\textnormal{Fr}_p^n$ to $\lambda$.  

\begin{coro}
 Let $n\geq 1$.  Every smooth $\fpb$-character of $\gal_{\mathbb{Q}_{p^n}}$ is of the form $\mu_{n,\lambda}\omega_n^{r}$, where $\lambda\in\fpb^\times$ and $0\leq r < p^n - 1$.  Moreover, the reciprocity maps $\iota_n$ induce a bijection between smooth $\fpb$-characters of $\gal_{\mathbb{Q}_{p^n}}$ and $\mathbb{Q}_{p^n}^\times$, given explicitly by $\mu_{n,\lambda}\omega_n^r\circ\iota_n = \mu_{\lambda^{-1}}\omega^r$.  
\end{coro}

\begin{proof}
 This follows from Lemmas \ref{firstgal} and \ref{brlemma}.  
\end{proof}

\vspace{\baselineskip}
\subsection{$L$-groups}  We now review the definition of the $L$-group of $G = \textnormal{U}(1,1)(\qpp/\qp)$.  For the general construction of $L$-groups, the reader should consult \cite{Bo79}; for the specific case of unitary groups, see Appendix A of \cite{BC09}, \cite{Mi11}, or Section 1.8 of \cite{Ro90}.  

Let $\widehat{G}$ denote the $\fpb$-valued points of the dual group of $G$; since $G$ splits over $\qpp$, we have $\widehat{G} = \textrm{GL}_2(\fpb)$.  We also define
$$\Phi_2 := \begin{pmatrix}\phantom{-}0 & 1\\ -1 & 0\end{pmatrix}\in\widehat{G}.$$

\begin{defi}
 The \emph{$L$-group of $G$} is defined as the semidirect product $${}^LG = \widehat{G}\rtimes\gal_{\qp} = \textnormal{GL}_2(\fpb)\rtimes\gal_{\qp},$$ with the action of $\gal_{\qp}$ on $\widehat{G}$ given by 
\begin{eqnarray*}
\textnormal{Fr}_pg\textnormal{Fr}_p^{-1} & = & \Phi_2(g^\top)^{-1}\Phi_2^{-1} = g\cdot\det(g)^{-1},\\
 hgh^{-1} & = & g,
\end{eqnarray*}
for $g\in \widehat{G}$, $h\in \gal_{\qpp}$.  
\end{defi}

In addition to the group $G$, we shall also need unitary groups of lower rank.  In particular, we will need the endoscopic group associated to $G$.  For the general definition in the complex case, see Sections 4.2 and 4.6 of \cite{Ro90}.

\begin{defi}
\textnormal{(a)} The group $J := (\textnormal{U}(1)\times\textnormal{U}(1))(\qpp/\qp)$ is the unique \emph{elliptic endoscopic group} associated to $G = \textnormal{U}(1,1)(\qpp/\qp)$.  

\noindent \textnormal{(b)} The \emph{$L$-group of $J$} is defined as the semidirect product $${}^LJ = (\fpb^\times\times\fpb^\times)\rtimes\gal_{\qp},$$ with the action of $\gal_{\qp}$ on $\fpb^\times\times\fpb^\times$ given by 
\begin{eqnarray*}
\textnormal{Fr}_p(x,y)\textnormal{Fr}_p^{-1} & = & (x^{-1},y^{-1}),\\
 h(x,y)h^{-1} & = & (x,y),
\end{eqnarray*}
for $x,y\in \fpb^\times$, $h\in \gal_{\qpp}$.  
\end{defi}

\begin{prop}[\cite{Ro90}, Proposition 4.6.1]\label{Lembedding}
 There exists a homomorphism $$\xi:{}^LJ\longhookrightarrow {}^LG,$$ which commutes with the projections to $\gal_{\qp}$, given by 
\begin{eqnarray*}
 (x,y) & \longmapsto & \begin{pmatrix}x & 0 \\ 0 & y\end{pmatrix},\\
 (1,1)\textnormal{Fr}_p & \longmapsto & \begin{pmatrix}0 & -1 \\ 1 & \phantom{-}0\end{pmatrix}\textnormal{Fr}_p,\\
 (1,1)h & \longmapsto & \begin{pmatrix}\mu_{2,-1}(h) & 0 \\ 0 & \mu_{2,-1}(h)\end{pmatrix}h,
\end{eqnarray*}
where $x,y\in\fpb^\times$, $h\in\gal_{\qpp}$.
\end{prop}

\vspace{\baselineskip}
\subsection{Langlands Parameters for $(\textnormal{U}(1)\times\textnormal{U}(1))(\qpp/\qp)$}  We begin by defining and investigating Langlands parameters in characteristic $p$ associated to $(\textnormal{U}(1)\times\textnormal{U}(1))(\qpp/\qp)$.  

\begin{defi}
A \emph{Langlands parameter} is a homomorphism 
$$\varphi: \gal_{\qp}\longrightarrow {}^LJ = (\fpb^\times\times\fpb^\times)\rtimes \gal_{\qp},$$
such that the composition of $\varphi$ with the canonical projection ${}^LJ\longrightarrow \gal_{\qp}$ is the identity map of $\gal_{\qp}$.  We say two Langlands parameters are \emph{equivalent} if they are conjugate by an element of $\fpb^\times\times\fpb^\times$.  
\end{defi}

With this definition, we come to our first result.  

\begin{prop}\label{u1params}
 Let $\varphi:\gal_{\qp}\longrightarrow {}^LJ$ be a Langlands parameter.  Then there exist $0\leq k,\ell < p + 1$ such that $\varphi$ is equivalent to the Langlands parameter $\eta_{k,\ell}$, defined by
\begin{center}
\begin{tabular}{ccc}
 $\eta_{k,\ell}(\textnormal{Fr}_p)$ & $=$ & $(1,1)\textnormal{Fr}_p$\\
 $\eta_{k,\ell}(h)$ & $=$ & $(\omega_2^{(1 - p)k}(h),\omega_2^{(1 - p)\ell}(h))h$,
\end{tabular}
\end{center}
where $h\in \gal_{\qpp}$.  
\end{prop}

\begin{proof}
 The conjugation action of $\gal_{\qp}$ on $\fpb^\times\times\fpb^\times$ shows that, up to equivalence, we have $\varphi(\textnormal{Fr}_p) = (1,1)\textnormal{Fr}_p$.  It remains to determine the image of $\gal_{\qpp}$.  Since $\gal_{\qpp}$ acts trivially on $\fpb^\times\times\fpb^\times$, we see that the restriction of $\varphi$ to $\gal_{\qpp}$ must be of the form
$$\varphi(h)  = (\mu_{2,\lambda_1}\omega_2^{r_1}(h),\mu_{2,\lambda_2}\omega_2^{r_2}(h))h,$$
where $\lambda_1,\lambda_2\in \fpb^\times$, $0\leq r_1,r_2 < p^2 - 1$, and $h\in\gal_{\qpp}$.  Lemma \ref{brlemma} and the definition of ${}^LJ$ imply 
\begin{eqnarray*}
 (\lambda_1,\lambda_2)\textnormal{Fr}_p^2 & = & \varphi(\textnormal{Fr}_p^2)\\
 & = & \varphi(\textnormal{Fr}_p)^2\\
 & = & (1,1)\textnormal{Fr}_p^2,\\
 (\omega_2^{pr_1}(h), \omega_2^{pr_2}(h))\textnormal{Fr}_ph\textnormal{Fr}_p^{-1} & = & \varphi(\textnormal{Fr}_ph\textnormal{Fr}_p^{-1})\\
 & = & \varphi(\textnormal{Fr}_p)\varphi(h)\varphi(\textnormal{Fr}_p)^{-1}\\
 & = & \textnormal{Fr}_p(\omega_2^{r_1}(h), \omega_2^{r_2}(h))h\textnormal{Fr}_p^{-1}\\
 & = & (\omega_2^{-r_1}(h), \omega_2^{-r_2}(h))\textnormal{Fr}_ph\textnormal{Fr}_p^{-1},
\end{eqnarray*}
which gives the result.  
\end{proof}

\begin{coro}\label{u1llc}
 There is a bijection between the Langlands parameters associated to the group $J$ and smooth irreducible representations of $(\textnormal{U}(1)\times\textnormal{U}(1))(\qpp/\qp)$, given explicitly by 
$$\eta_{k,\ell} \longleftrightarrow \omega^k\otimes\omega^\ell,$$
where $0\leq k,\ell < p + 1$, and $\omega$ is the character defined in equation \eqref{omega}.
\end{coro}

\vspace{\baselineskip}
\subsection{Langlands Parameters for $\textnormal{U}(1,1)(\qpp/\qp)$} We now proceed to explore Langlands parameters in characteristic $p$ for the group $G = \textnormal{U}(1,1)(\qpp/\qp)$.  For the analogous definitions in the complex setting, see \cite{Ro90}, \cite{Ro92}, and Appendix A of \cite{BC09}.  
\begin{defi}\label{params}
 \textnormal{(a)} A \emph{Langlands parameter} is a homomorphism 
$$\varphi: \gal_{\qp}\longrightarrow {}^LG = \textnormal{GL}_2(\fpb)\rtimes\gal_{\qp},$$
such that the composition of $\varphi$ with the canonical projection ${}^LG\longrightarrow \gal_{\qp}$ is the identity map of $\gal_{\qp}$.  We say two Langlands parameters are \emph{equivalent} if they are conjugate by an element of $\widehat{G} = \textnormal{GL}_2(\fpb)$.  

\noindent \textnormal{(b)} Let $\varphi:\gal_{\qp}\longrightarrow {}^LG$ be a Langlands parameter and let $0\leq k < p + 1$.  We define the \emph{twist of $\varphi$ by $\omega_2^{(1 - p)k}$}, denoted $\varphi\otimes\omega_2^{(1 - p)k}$, by 
\begin{center}
\begin{tabular}{ccc}
 $\varphi\otimes\omega_2^{(1 - p)k}(h)$ & = & $\varphi(h)\begin{pmatrix}\omega_2^{(1 - p)k}(h) & 0 \\ 0 & \omega_2^{(1 - p)k}(h)\end{pmatrix}$,\\
 $\varphi\otimes\omega_2^{(1 - p)k}(\textnormal{Fr}_ph)$ & = & $\varphi(\textnormal{Fr}_ph)\begin{pmatrix}\omega_2^{(1 - p)k}(h) & 0 \\ 0 & \omega_2^{(1 - p)k}(h)\end{pmatrix}$,
\end{tabular}
\end{center}
where $h\in\gal_{\qpp}$.  One easily checks that this is well-defined and gives a bona fide Langlands parameter.  
\end{defi}

\begin{defi}
 Let $\varphi:\gal_{\qp}\longrightarrow {}^LG$ be a Langlands parameter.  Since the group $\gal_{\qpp}$ acts trivially on $\widehat{G}$, the restriction of $\varphi$ to $\gal_{\qpp}$ must be of the form $$\varphi(h) = \varphi_0(h)h,$$ where $h\in\gal_{\qpp}$ and $\varphi_0:\gal_{\qpp}\longrightarrow \widehat{G}$ is a \emph{homomorphism}.  As $\widehat{G} = \textnormal{GL}_2(\fpb)$, $\varphi_0$ is a two-dimensional Galois representation; we call it the \emph{Galois representation associated to $\varphi$}.  
\end{defi}

\begin{defi}
 Let $\varphi: \gal_{\qp}\longrightarrow {}^LG$ be a Langlands parameter.  We say $\varphi$ is \emph{stable} if the associated Galois representation $\varphi_0:\gal_{\qpp}\longrightarrow \textnormal{GL}_2(\fpb)$ is irreducible.
\end{defi}

\vspace{\baselineskip}

Our first result on Langlands parameters for $\textnormal{U}(1,1)(\qpp/\qp)$ stands in stark contrast to the complex case (cf. \cite{Ro90}, Section 15.1).

\begin{prop}\label{nostables}
 There do not exist any stable parameters $\varphi:\gal_{\qp}\longrightarrow {}^LG$.  
\end{prop}

\begin{proof}
The inclusion $\textnormal{SL}_2(\qp)\longhookrightarrow \textnormal{U}(1,1)(\qpp/\qp)$ gives rise, by duality, to a homomorphism of $L$-groups, given explicitly by
\begin{eqnarray*}
 \imath :{}^LG = \textnormal{GL}_2(\fpb)\rtimes \gal_{\qp} & \longrightarrow & {}^L\textnormal{SL}_2 = \textnormal{PGL}_2(\fpb)\times\gal_{\qp}\\
 gh & \longmapsto & \llbracket g\rrbracket h,
\end{eqnarray*}
where $g\in \textnormal{GL}_2(\fpb), h\in \gal_{\qp}$, and $\llbracket g\rrbracket$ denotes the image in $\textnormal{PGL}_2(\fpb)$ of the element $g$.  Given $\varphi:\gal_{\qp} \longrightarrow {}^LG$, we consider the Langlands parameter $\imath\circ \varphi:\gal_{\qp} \longrightarrow\textnormal{PGL}_2(\fpb)\times\gal_{\qp}$ and the associated Galois representation $(\imath\circ\varphi)_0:\gal_{\qpp}\longrightarrow \textnormal{PGL}_2(\fpb)$.  It is clear that $\varphi_0$ is irreducible if and only if $(\imath\circ\varphi)_0$ is irreducible, and therefore it suffices to examine $\imath\circ\varphi$.  

Now, since the group $\textnormal{SL}_2(\qp)$ is split over $\qp$, $\gal_{\qp}$ acts trivially on $\widehat{\textnormal{SL}_2} = \textnormal{PGL}_2(\fpb)$, and therefore $\imath\circ\varphi$ takes the form
$$\imath\circ\varphi(h) = \varphi'(h)h,$$
where $h\in \gal_{\qp}$ and $\varphi':\gal_{\qp}\longrightarrow \textnormal{PGL}_2(\fpb)$ is a homomorphism.  Assume $\varphi'$ is irreducible.  By a theorem of Tate (see the proof of Theorem 4 (and its corollary) in \cite{Se77b}), we have
$$\textnormal{H}^2(\gal_{\qp},\fpb^\times) = 0,$$
which implies that every projective representation has a lift to $\textnormal{GL}_2(\fpb)$.  Hence, we may write $\varphi'(h) = \llbracket \widetilde{\varphi}'(h)\rrbracket$, where $\widetilde{\varphi}':\gal_{\qp}\longrightarrow \textnormal{GL}_2(\fpb)$ is an irreducible Galois representation.

It is well-known (\cite{Vig97}, Section 1.14) that every two-dimensional irreducible mod-$p$ representation of $\gal_{\qp}$ is isomorphic to a representation of the form
$$\textnormal{ind}_{\gal_{\qpp}}^{\gal_{\qp}}(\mu_{2,\lambda}\omega_2^m),$$
where $\lambda\in\fpb^\times$, and $0\leq m < p^2 - 1$ satisfies $m\not\equiv pm~(\textnormal{mod}~ p^2 - 1)$.  By Mackey theory, we have 
$$\widetilde{\varphi}'|_{\gal_{\qpp}} \cong \textnormal{ind}_{\gal_{\qpp}}^{\gal_{\qp}}(\mu_{2,\lambda}\omega_2^m)|_{\gal_{\qpp}} \cong \mu_{2,\lambda}\omega_2^m\oplus\mu_{2,\lambda}\omega_2^{pm},$$
which implies that the original Langlands parameter $\varphi$ cannot be stable.  
\end{proof}

\vspace{\baselineskip}
Using the parameters $\eta_{k,\ell}$ above, we obtain the first nontrivial (necessarily nonstable) examples of Langlands parameters for the group $G$.

\begin{defi}\label{endparams}
 Let $0\leq k,\ell < p+1$.  We denote by $\varphi_{k,\ell}:\gal_{\qp}\longrightarrow {}^LG$ the Langlands parameter obtained by composing $\eta_{k,\ell}$ (of Proposition \ref{u1params}) with $\xi$ (of Proposition \ref{Lembedding}).  Explicitly, we have
\begin{center}
\begin{tabular}{ccc}
 $\varphi_{k,\ell}(\textnormal{Fr}_p)$ & $=$ & $\begin{pmatrix}0 & -1\\ 1 & \phantom{-}0\end{pmatrix}\textnormal{Fr}_p$,\\
 $\varphi_{k,\ell}(h)$ & $=$ & $\begin{pmatrix}\mu_{2,-1}\omega_2^{(1 - p)k}(h) & 0\\ 0 & \mu_{2,-1}\omega_2^{(1 - p)\ell}(h)\end{pmatrix}h$,
\end{tabular}
\end{center}
for $h\in\gal_{\qpp}$.  

We say $\varphi_{k,\ell}$ is \emph{regular} if $k\neq \ell$, and \emph{singular} otherwise.  
\end{defi}

\begin{lemm}\label{paramequiv}
 Let $0\leq k, k', \ell, \ell' < p + 1$.  Then $\varphi_{k,\ell}$ is equivalent to $\varphi_{k',\ell'}$ if and only if the sets $\{k,\ell\}$ and $\{k',\ell'\}$ coincide.  
\end{lemm}

\begin{proof}
This is left as an easy exercise.  
\end{proof}

\begin{coro}\label{scbij}
 There exists a bijection between $\widehat{G}$-equivalence classes of regular Langlands parameters coming from the endoscopic group $J = (\textnormal{U}(1)\times\textnormal{U}(1))(\qpp/\qp)$ and $L$-packets of irreducible supercuspidal representations on the group $G = \textnormal{U}(1,1)(\qpp/\qp)$, given by
$$\varphi_{k,\ell}\longleftrightarrow \{\cusp{\ell}{[k - \ell - 1]},~\cusp{k}{[\ell - k - 1]}\},$$
where $0\leq k,\ell < p + 1$, and where $[k - \ell - 1]$ (resp. $[\ell - k - 1]$) denotes the unique integer between $0$ and $p - 1$ equivalent to $k - \ell - 1$ (resp. $\ell - k - 1$) modulo $p + 1$.  Moreover, this bijection is compatible with twisting by characters on both sides (under the one-dimensional version of the correspondence of Corollary \ref{u1llc}).
\end{coro}

\begin{proof}
Let $\Pi_{k,\ell}$ denote the $L$-packet on the right-hand side of the correspondence above.  Proposition \ref{Lpackets} and Lemma \ref{paramequiv} show that $\Pi_{k,\ell}$ and $\Pi_{k',\ell'}$ are identical if and only if $\{k,\ell\} = \{k', \ell'\}$, if and only if $\varphi_{k,\ell}$ is equivalent to $\varphi_{k', \ell'}$.  
\end{proof}

\vspace{\baselineskip}

Our next task will be to extend the correspondence of the above corollary to nonsupercuspidal $L$-packets.  In doing so, we are led to consider Langlands parameters arising from a proper Levi subgroup of ${}^LG$.  We let 
$${}^LT := \widehat{T}\rtimes \gal_{\qp},$$
where $\widehat{T}$ is the diagonal maximal torus of $\widehat{G}$, and the action of $\gal_{\qp}$ on $\widehat{T}$ is the restriction of the action on $\widehat{G}$.  

\begin{prop}\label{paramsfromT}
 Let $\varphi:\gal_{\qp}\longrightarrow {}^LG$ be a Langlands parameter which factors through the group ${}^LT$, that is, such that $\varphi$ is the composition 
$$\gal_{\qp}\longrightarrow {}^LT\longhookrightarrow {}^LG,$$
where the second arrow denotes the canonical inclusion.  Then there exist $0\leq r < p^2 - 1$ and $\lambda\in\fpb^\times$ such that $\varphi$ is equivalent to the Langlands parameter $\psi_{r,\lambda}$, defined by
\begin{center}
\begin{tabular}{ccc}
 $\psi_{r,\lambda}(\textnormal{Fr}_p)$ & $=$ & $\begin{pmatrix}1 & 0 \\ 0 & \lambda\end{pmatrix}\textnormal{Fr}_p$\\
 $\psi_{r,\lambda}(h)$ & $=$ & $\begin{pmatrix}\mu_{2,\lambda^{-1}}\omega_2^r(h) & 0 \\ 0 & \mu_{2,\lambda}\omega_2^{-pr}(h) \end{pmatrix}h$,
\end{tabular}
\end{center}
where $h\in \gal_{\qpp}$.  
\end{prop}

\begin{proof}
 We proceed as in the proof of Proposition \ref{u1params}.  Using the action of $\gal_{\qp}$ on $\widehat{T}$ we may assume that, up to equivalence, we have
$$\varphi(\textnormal{Fr}_p) = \begin{pmatrix}1 & 0 \\ 0 & \lambda\end{pmatrix}\textnormal{Fr}_p$$
for some $\lambda\in \fpb^\times$.  Let $\varphi_0:\gal_{\qpp}\longrightarrow \widehat{T}\longhookrightarrow \widehat{G}$ be the Galois representation associated to $\varphi$, so that
$$\varphi(h) = \varphi_0(h)h = \begin{pmatrix}\mu_{2,\lambda_1}\omega_2^{r_1}(h) & 0 \\ 0 & \mu_{2,\lambda_2}\omega_2^{r_2}(h) \end{pmatrix}h,$$
where $\lambda_1, \lambda_2\in\fpb^\times$, $0\leq r_1, r_2 < p^2 - 1$, and $h\in\gal_{\qpp}$.  Again using Lemma \ref{brlemma} and the definition of ${}^LT$, we obtain
\begin{eqnarray*}
 \begin{pmatrix}\lambda_1 & 0 \\ 0 & \lambda_2\end{pmatrix}\textnormal{Fr}_p^2 & = & \varphi(\textnormal{Fr}_p^2)\\
 & = & \varphi(\textnormal{Fr}_p)^2\\
 & = & \begin{pmatrix}1 & 0 \\ 0 & \lambda\end{pmatrix}\textnormal{Fr}_p\begin{pmatrix}1 & 0 \\ 0 & \lambda\end{pmatrix}\textnormal{Fr}_p\\
 & = & \begin{pmatrix}\lambda^{-1} & 0 \\ 0 & \lambda\end{pmatrix}\textnormal{Fr}_p^2,\\
 \begin{pmatrix}\omega_2^{pr_1}(h) & 0 \\ 0 & \omega_2^{pr_2}(h)\end{pmatrix}\textnormal{Fr}_ph\textnormal{Fr}_p^{-1} & = & \varphi(\textnormal{Fr}_ph\textnormal{Fr}_p^{-1})\\
 & = & \varphi(\textnormal{Fr}_p)\varphi(h)\varphi(\textnormal{Fr}_p)^{-1}\\
 & = & \begin{pmatrix}\omega_2^{-r_2}(h) & 0 \\ 0 & \omega_2^{-r_1}(h)\end{pmatrix}\textnormal{Fr}_ph\textnormal{Fr}_p^{-1},
\end{eqnarray*}
which gives the result.  
\end{proof}

\vspace{\baselineskip}
We shall also need more precise information about equivalence classes of the Langlands parameters $\varphi_{k,\ell}$ and $\psi_{r,\lambda}$.  This is the content of the following two lemmas, whose proofs are left as exercises for the reader.  

\begin{lemm}\label{paramequiv2}
  Let $0\leq r,r'<p^2 - 1$ and $\lambda, \lambda'\in \fpb^\times$.  Then $\psi_{r,\lambda}$ is equivalent to $\psi_{r',\lambda'}$ if and only if $r' = r, \lambda' = \lambda$ or $r' \equiv -pr~(\textnormal{mod}~p^2 - 1), \lambda' = \lambda^{-1}.$
\end{lemm}

\begin{lemm}\label{endvsnonend}\hfill
\begin{enumerate}[(a)]
\item Assume $p\neq 2$.  Let $0\leq k,\ell < p+1,~~~$ $0\leq r < p^2 - 1$, and $\lambda\in\fpb^\times$.  Then $\varphi_{k,\ell}$ is equivalent to $\psi_{r,\lambda}$ if and only if $k = \ell$, $r\equiv (1 - p)k~(\textnormal{mod}~p^2 - 1),$ and $\lambda = -1$.  
\item Assume $p = 2$.  Then there are no equivalences between parameters $\varphi_{k,\ell}$ and $\psi_{r,\lambda}$.  
 \end{enumerate}
\end{lemm}

\vspace{\baselineskip}

\begin{defi}\label{corr}
 We define a \emph{``semisimple mod-$p$ correspondence for $G = \textnormal{U(1,1)}(\qpp/\qp)$''} to be the following correspondence between certain $\widehat{G}$-equivalence classes of Langlands parameters over $\fpb$ and certain isomorphism classes of semisimple $L$-packets on $\textnormal{U(1,1)}(\qpp/\qp)$:
\begin{itemize}
 \item \emph{the supercuspidal case:}
Let $0\leq k,\ell<p + 1$ with $k\neq \ell$.
$$\boxed{\varphi_{k,\ell} \longleftrightarrow \left\{\cusp{\ell}{[k - \ell - 1]},~\cusp{k}{[\ell - k - 1]}\right\}}$$
\vspace{\baselineskip}
 \item \emph{the nonsupercuspidal case:}
Let $0\leq r \leq p - 1$, $\lambda\in\fpb^\times$, and $0\leq k < p+1$.  
\begin{itemize}
\item if $(r,\lambda)\neq (0,1),~(p-1,1)$:

\vspace{.5\baselineskip}
\begin{center}
\fbox{$\begin{array}{l}
\psi_{r,\lambda}\otimes\omega_2^{(1-p)k} = \psi_{r + (1-p)k,\lambda}\\ 
\qquad\longleftrightarrow \left\{(\omega^k\circ\det)\otimes\textnormal{ind}_B^G(\mu_{\lambda^{-1}}\omega^{-pr})~\oplus~(\omega^k\circ\det)\otimes\textnormal{ind}_B^G(\mu_{\lambda}\omega^{r})\right\}
 \end{array}$}
\end{center}
\vspace{.5\baselineskip}

\item if $(r,\lambda) = (0,1)$:
\vspace{.5\baselineskip}
\begin{center}
\fbox{$\begin{array}{l}
\psi_{0,1}\otimes\omega_2^{(1-p)k} = \psi_{(1-p)k,1}\\ \qquad\longleftrightarrow \left\{\omega^k\circ\det~\oplus~(\omega^k\circ\det)\otimes\textnormal{St}_G~\oplus~\omega^k\circ\det~\oplus~(\omega^k\circ\det)\otimes\textnormal{St}_G\right\}
\end{array}$}
\end{center}
\vspace{.5\baselineskip}

\end{itemize}
\end{itemize}
\end{defi}

\vspace{2\baselineskip}

\subsection{Remarks}
\noindent \textnormal{(1)}~ Corollary \ref{scbij} and Lemmas \ref{paramequiv} and \ref{paramequiv2} imply that the correspondence above is well-defined.  

\vspace{\baselineskip}
\noindent \textnormal{(2)}~ We may state this correspondence more elegantly as follows.  For $0\leq r\leq p - 1$, the group $K$ acts irreducibly on the representation $\sigma_r$ defined in Subsection \ref{gl2sc}.  We let $\tau_{r,1}$ denote the endomorphism of $\textnormal{c-ind}_K^G(\sigma_r)$ which corresponds via Frobenius Reciprocity to the function with support $K\alpha^{-1}K$ and taking the value $U_r$ at $\alpha^{-1}$.   Here $U_r$ is the endomorphism of $\sigma_r$ given by 
$$U_r.x^{r - i}y^i = \begin{cases}0 & \textnormal{if}~i\neq r,\\ y^r & \textnormal{if}~i = r,\end{cases}$$
and 
$$\alpha = \begin{pmatrix}\varpi^{-1} & 0 \\ 0 & \varpi \end{pmatrix}.$$
The spherical Hecke algebra $\hh_{\fpb}(G,K,\sigma_r)$ of $G$-equivariant endomorphisms of the compactly induced representation $\textnormal{c-ind}_K^G(\sigma_r)$ is then isomorphic to a polynomial algebra over $\fpb$ in one variable, generated by an endomorphism $\tau_r$.  Explicitly, we have
$$\tau_r = \begin{cases}\tau_{r,1}& \textnormal{if}~r \neq 0,\\ \tau_{r,1} + 1 & \textnormal{if}~ r = 0,\end{cases}$$
(this definition comes from the Satake isomorphism).  For $\lambda\in\fpb^\times$, we define
$$\bpi(r,\lambda):=\frac{\textnormal{c-ind}_K^G(\sigma_r)}{(\tau_r - \lambda)}.$$
A simple argument shows that 
$$\bpi(r,\lambda)|_{G_\s}\cong \pi_0(r,\lambda),$$
where $\pi_0(r,\lambda)$ denotes the representation of $\textnormal{SL}_2(\qp)$ (viewed as a representation of $G_\s$) defined in \cite{Ab12}, Section 3.4.  Using Th\'eor\`eme 3.18 (\emph{loc. cit.}) and the existence of certain $I(1)$-invariant elements of $\bpi(r,\lambda)$ (along with Proposition \ref{nonscres}), we deduce
$$\bpi(r,\lambda)\cong \begin{cases}\textnormal{ind}_B^G(\mu_{\lambda^{-1}}\omega^{-pr})& \textnormal{if}~(r,\lambda)\neq (0,1),\\ \textnormal{nonsplit extension of}~ 1_G~\textnormal{by}~\textnormal{St}_G& \textnormal{if}~(r,\lambda) = (0,1). \end{cases}$$

If we let $\pi^{\textnormal{ss}}$ denote the semisimplification of a smooth representation $\pi$ of $G$, we obtain
$$\bpi(r,\lambda)^\textnormal{ss} = \begin{cases}\textnormal{ind}_B^G(\mu_{\lambda^{-1}}\omega^{-pr}) & \textnormal{if}~(r,\lambda)\neq (0,1),~(p-1,1),\\ \omega^{p}\circ\det~\oplus~(\omega^{p}\circ\det)\otimes\textnormal{St}_G & \textnormal{if}~(r,\lambda) = (p - 1,1),\\ 1_G~\oplus~\textnormal{St}_G & \textnormal{if}~(r,\lambda) = (0,1).\end{cases}$$

\vspace{\baselineskip}

The correspondence of Definition \ref{corr} now takes the form:

\begin{defnI}\hfill\newline
\noindent $\bullet$ \emph{The supercuspidal case:}
Let $0\leq k,\ell<p + 1$ with $k\neq \ell$.
$$\boxed{\varphi_{k,\ell} \longleftrightarrow \left\{\cusp{\ell}{[k - \ell - 1]},~\cusp{k}{[\ell - k - 1]}\right\}}$$
\noindent $\bullet$ \emph{The nonsupercuspidal case:}
Let $0\leq r \leq p - 1$, $\lambda\in\fpb^\times$, and $0\leq k < p+1$.
\begin{center}
\fbox{$\begin{array}{l}
 \psi_{r,\lambda}\otimes\omega_2^{(1 - p)k} = \psi_{r + (1 - p)k,\lambda} \\
 \longleftrightarrow \{(\omega^k\circ\det)\otimes\bpi(r,\lambda)^{\textnormal{ss}}~\oplus~(\omega^{k + r + 1}\circ\det)\otimes\bpi(p - 1 - r,\lambda^{-1})^{\textnormal{ss}}\}
\end{array}$}
\end{center}
\end{defnI}

\vspace{\baselineskip}
\noindent \textnormal{(3)}~ Suppose $p\neq 2$.  The correspondences of Corollary \ref{u1llc} and Definition \ref{corr}, along with the homomorphism $\xi$ of Proposition \ref{Lembedding}, imply that we have an \emph{endoscopic transfer map}
$$\widetilde{\xi}:\mathfrak{Irr}_{\fpb}((\textnormal{U}(1)\times\textnormal{U}(1))(\qpp/\qp))\longrightarrow \mathfrak{L}\textnormal{-}\mathfrak{pack}_{\fpb}(\textnormal{U}(1,1)(\qpp/\qp))$$
from the set of isomorphism classes of smooth irreducible representations of $(\textnormal{U}(1)\times\textnormal{U}(1))(\qpp/\qp)$ to the set of isomorphism classes of $L$-packets of semisimple representations on $\textnormal{U}(1,1)(\qpp/\qp)$.  Using Lemma \ref{endvsnonend}, the map $\widetilde{\xi}$ is given explicitly by
$$\widetilde{\xi}\left(\omega^k\otimes\omega^\ell\right) = \begin{cases} \left\{\cusp{\ell}{[k - \ell - 1]},~\cusp{k}{[\ell - k - 1]}\right\} & \textnormal{if}~ k \neq \ell,\\ \left\{\textnormal{ind}_B^G(\mu_{-1}\omega^{(1 - p)k})~\oplus~\textnormal{ind}_B^G(\mu_{-1}\omega^{(1 - p)k})\right\} & \textnormal{if}~k = \ell.\end{cases}$$
This bears a striking resemblence to the complex case (see Proposition 11.1.1 of \cite{Ro90}, especially points (c) and (e)).  Moreover, the equation
$$\widetilde{\xi}\left(\omega^0\otimes\omega^0\right) = \left\{\textnormal{ind}_B^G(\mu_{-1})~\oplus~\textnormal{ind}_B^G(\mu_{-1})\right\}$$
gives an example of \emph{transfer of unramified representations}, which may also be deduced from (a modified version of) the discussion in Section 2.7 of \cite{Mi11} (see also Theorem 4.4, \emph{loc. cit.}, and Section 4.5 of \cite{Ro90}).

\appendix
\section{Relation to $C$-groups}

We now translate the results of the previous section into the language of $C$-groups of Buzzard--Gee \cite{BG13}.  As most of the computations are similar to those already given in the case of $L$-groups, we will omit them.  We refer to \cite{BG13} throughout, in particular drawing on the example contained in Section 8.3.  Additionally, we assume throughout that $p\neq 2$.

The $C$-group of $G$, denoted ${}^CG$, is defined as ${}^L\widetilde{G}$, where $\widetilde{G}$ is an algebraic group defined by a certain central $\gm$-extension of the algebraic group defining $G$:
$$1\longrightarrow \gm \longrightarrow \widetilde{G}\longrightarrow \mathbf{U}(1,1) \longrightarrow 1.$$
For the precise definition, see Proposition 5.3.1 of \cite{BG13}.  By Proposition 5.3.3 (\emph{loc. cit.}), the $\fpb$-points of the dual group $\widehat{\widetilde{G}}$ take the form
$$\widehat{\widetilde{G}} \cong (\widehat{G}\times\fpb^\times)/\langle(-\textnormal{id},-1)\rangle = (\textnormal{GL}_2(\fpb)\times\fpb^\times)/\langle(-\textnormal{id},-1)\rangle.$$
We will denote elements of $\widehat{\widetilde{G}}$ by $[g,\mu]$, with $g\in \textnormal{GL}_2(\fpb)$, $\mu\in\fpb^\times$.  The action of $\gal_{\qp}$ on $\widehat{\widetilde{G}}$ is the one induced from its action on $\widehat{G}$.  Therefore, we see that the $C$-group is given by the semidirect product
$${}^CG = \widehat{\widetilde{G}}\rtimes \gal_{\qp} = \left((\textnormal{GL}_2(\fpb)\times\fpb^\times)/\langle(-\textnormal{id},-1)\rangle\right)\rtimes\gal_{\qp},$$
with the action of $\gal_{\qp}$ on $\widehat{\widetilde{G}}$ given by
\begin{eqnarray*}
\textnormal{Fr}_p[g,\mu]\textnormal{Fr}_p^{-1} & = & [\Phi_2(g^\top)^{-1}\Phi_2^{-1},\mu],\\
 h[g,\mu]h^{-1} & = & [g,\mu],
\end{eqnarray*}
for $g\in \widehat{G}$, $\mu\in\fpb^\times$, $h\in \gal_{\qpp}$.

The inclusion $\gm\longrightarrow \widetilde{G}$ induces, by duality, a map $d:{}^CG\longrightarrow \fpb^\times$, which is given explicitly by
\begin{center}
 \begin{tabular}{rccl}
  $d:$ & ${}^CG$ & $\longrightarrow$ & $\fpb^\times$\\
 & $[g,\mu]\textnormal{Fr}_p$ & $\longmapsto$ &  $\mu^2$,\\
 & $[g,\mu]h$ & $\longmapsto$ &  $\mu^2$,
 \end{tabular}
\end{center}
where $g\in \widehat{G}, \mu\in \fpb^\times$, and $h\in \gal_{\qpp}$.

We now consider Langlands parameters with target ${}^CG$.  

\begin{defi}
 \textnormal{(a)} A \emph{Langlands parameter} is a homomorphism 
$${}^C\varphi: \gal_{\qp}\longrightarrow {}^CG = \widehat{\widetilde{G}}\rtimes\gal_{\qp},$$
such that the composition of ${}^C\varphi$ with the canonical projection ${}^CG\longrightarrow \gal_{\qp}$ is the identity map of $\gal_{\qp}$.  We say two Langlands parameters are \emph{equivalent} if they are conjugate by an element of $\widehat{\widetilde{G}}$.  

\noindent \textnormal{(b)} Let ${}^C\varphi:\gal_{\qp}\longrightarrow {}^CG$ be a Langlands parameter.  Since the group $\gal_{\qpp}$ acts trivially on $\widehat{\widetilde{G}}$, the restriction of ${}^C\varphi$ to $\gal_{\qpp}$ must be of the form 
$${}^C\varphi(h) = {}^C\varphi_0(h)h,$$ where $h\in\gal_{\qpp}$ and ${}^C\varphi_0:\gal_{\qpp}\longrightarrow \widehat{\widetilde{G}}$ is a \emph{homomorphism}.  We say ${}^C\varphi$ is \emph{stable} if the image of the associated Galois representation ${}^C\varphi_0:\gal_{\qpp}\longrightarrow \widehat{\widetilde{G}}$ is not contained in any proper parabolic subgroup of $\widehat{\widetilde{G}}$.  
\end{defi}

\vspace{\baselineskip}

The first concrete examples of Langlands parameters with target ${}^CG$ are given by the following definition.  Note that the construction of $\widehat{\widetilde{G}}$ shows that the element $\omega_1^{1/2}$ appearing in the definition is unambiguous.  

\begin{defi}
 \textnormal{(a)} Let $0\leq k,\ell < p+1$.  We denote by ${}^C\varphi_{k,\ell}:\gal_{\qp}\longrightarrow {}^CG$ the following Langlands parameter:
\begin{center}
\begin{tabular}{ccc}
 ${}^C\varphi_{k,\ell}(\textnormal{Fr}_p)$ & $=$ & $\left[\begin{pmatrix}0 & -1\\ 1 & \phantom{-}0\end{pmatrix}, 1\right]\textnormal{Fr}_p$,\\
 ${}^C\varphi_{k,\ell}(h)$ & $=$ & $\left[\begin{pmatrix}\mu_{2,-1}\omega_2^{-1 + (1 - p)k}(h) & 0\\ 0 & \mu_{2,-1}\omega_2^{-1 + (1 - p)\ell}(h)\end{pmatrix}\omega_1^{1/2}(h),~\omega_1^{1/2}(h)\right]h$,
\end{tabular}
\end{center}
for $h\in\gal_{\qpp}$.  

\noindent \textnormal{(b)} Let $0\leq r < p^2 - 1$ and $\lambda\in\fpb^\times$.  We denote by ${}^C\psi_{r,\lambda}:\gal_{\qp}\longrightarrow {}^CG$ the following Langlands parameter:
\begin{center}
\begin{tabular}{ccc}
 ${}^C\psi_{r,\lambda}(\textnormal{Fr}_p)$ & $=$ & $\left[\begin{pmatrix}1 & 0\\ 0 & \lambda\end{pmatrix},1\right]\textnormal{Fr}_p$,\\
 ${}^C\psi_{r,\lambda}(h)$ & $=$ & $\left[\begin{pmatrix}\mu_{2,\lambda^{-1}}\omega_2^r(h) & 0 \\ 0 & \mu_{2,\lambda}\omega_2^{-pr - (p + 1)}(h) \end{pmatrix}\omega_1^{1/2}(h),~\omega_1^{1/2}(h)\right]h$,
\end{tabular}
\end{center}
for $h\in\gal_{\qpp}$.  
\end{defi}

\begin{lemm}

\begin{enumerate}[(a)]
\item  Let $0 \leq k,k',\ell,\ell' < p + 1$.  Then ${}^C\varphi_{k,\ell}$ is equivalent to ${}^C\varphi_{k',\ell'}$ if an only if the sets $\{k,\ell\}$ and $\{k',\ell'\}$ coincide.  

\item Let $0\leq r,r' < p^2 - 1$ and $\lambda,\lambda'\in \fpb^\times$.  Then ${}^C\psi_{r,\lambda}$ is equivalent to ${}^C\psi_{r',\lambda'}$ if and only if $r' = r$, $\lambda' = \lambda$ or $r' \equiv -pr - (p + 1)~(\textnormal{mod}~p^2 - 1), \lambda' = \lambda^{-1}$.  

\item Let $0\leq k,\ell < p + 1$, $0\leq r < p^2 - 1$ and $\lambda\in \fpb^\times$.  Then ${}^C\varphi_{k,\ell}$ is equivalent to ${}^C\psi_{r,\lambda}$ if and only if $k = \ell$, $r\equiv -1 + (1 - p)k~(\textnormal{mod}~p^2 - 1),$ and $\lambda = -1$.
\end{enumerate}
\end{lemm}

\begin{proof}
 This is left as an exercise.  
\end{proof}

We may now deduce the following results.

\begin{prop}
\begin{enumerate}[(a)]
\item There do not exist any stable parameters ${}^C\varphi:\gal_{\qp}\longrightarrow {}^CG$.  

\item Let ${}^C\varphi:\gal_{\qp}\longrightarrow {}^CG$ denote a Langlands parameter which is \emph{semisimple} (that is, for which ${}^C\varphi_0(h)$ is semisimple for every $h\in \gal_{\qpp}$), and for which $d\circ{}^C\varphi:\gal_{\qp}\longrightarrow \fpb^\times$ is equal to $\omega_1$ (cf. \cite{BG13}, Conjecture 5.3.4).  Then ${}^C\varphi$ is equivalent to either ${}^C\varphi_{k,\ell}$ or ${}^C\psi_{r,\lambda}$.  
\end{enumerate}
\end{prop}

\begin{proof}
 This is left as an exercise.
\end{proof}

\vspace{\baselineskip}

We may now state an analog of Definition \ref{corr}.  

\begin{defi}
 We define a \emph{``semisimple mod-$p$ correspondence for $G = \textnormal{U(1,1)}(\qpp/\qp)$''} to be the following correspondence between certain equivalence classes of ${}^CG$-valued Langlands parameters over $\fpb$ and certain isomorphism classes of semisimple $L$-packets on $G$:
\begin{itemize}
\item \emph{The supercuspidal case:}
Let $0\leq k,\ell<p + 1$ with $k\neq \ell$.
$$\boxed{{}^C\varphi_{k,\ell} \longleftrightarrow \left\{\cusp{\ell}{[k - \ell - 1]},~\cusp{k}{[\ell - k - 1]}\right\}}$$
\item \emph{The nonsupercuspidal case:}
Let $0\leq r \leq p - 1$, $\lambda\in\fpb^\times$, and $0\leq k < p+1$.
$$\boxed{{}^C\psi_{(r - 1) + (1 - p)k,\lambda}\longleftrightarrow \{(\omega^k\circ\det)\otimes\bpi(r,\lambda)^{\textnormal{ss}}~\oplus~(\omega^{k + r + 1}\circ\det)\otimes\bpi(p - 1 - r,\lambda^{-1})^{\textnormal{ss}}\}}$$  
\end{itemize}
\end{defi}

\nocite{*}
\bibliographystyle{amsplain}
\bibliography{U11bib}

\end{document}